\documentclass{amsart}

\usepackage[linktoc=page]{hyperref}

\usepackage[utf8]{inputenc}
\usepackage{cite}
\usepackage{amssymb}
\usepackage{amsthm}
\usepackage{amsfonts}
\usepackage{amsmath}
\usepackage{mathrsfs}
\usepackage[all]{xy}
\usepackage{graphicx}
\usepackage{mathrsfs}
\usepackage{extpfeil}
\usepackage{mathtools}
\usepackage{tikz-cd}
\usepackage{bm}
\usepackage{latexsym, amscd ,psfrag}
\usepackage{graphicx}
\usepackage{caption}
\usepackage{float}

\usepackage[mathscr]{euscript}
\usepackage{enumitem}
\usepackage{cleveref}

\makeatletter
\newsavebox{\@brx}
\newcommand{\llangle}[1][]{\savebox{\@brx}{\(\m@th{#1\langle}\)}%
  \mathopen{\copy\@brx\kern-0.5\wd\@brx\usebox{\@brx}}}
\newcommand{\rrangle}[1][]{\savebox{\@brx}{\(\m@th{#1\rangle}\)}%
  \mathclose{\copy\@brx\kern-0.5\wd\@brx\usebox{\@brx}}}
\makeatother

\newcommand{\si}{\sigma}
\newcommand{\Si}{\Sigma}
\newcommand{\Ga}{\Gamma}

\newcommand{\bC}{\mathbb{C}}

\newcommand{\bK}{\mathbb{K}}
\newcommand{\bL}{\mathbb{L}}
\newcommand{\bP}{\mathbb{P}}
\newcommand{\bQ}{\mathbb{Q}}
\newcommand{\bR}{\mathbb{R}}

\newcommand{\bZ}{\mathbb{Z}}
\newcommand{\btau}{\boldsymbol{\tau}}

\newcommand{\bSi}{{\boldsymbol \Si}}

\newcommand{\bq}{\mathbf{q}}

\newcommand{\cF}{\mathcal{F}}

\newcommand{\cL}{\mathcal{L}}

\newcommand{\cM}{\mathcal{M}}

\newcommand{\cR}{\mathcal{R}}
\newcommand{\cS}{\mathcal{S}}

\newcommand{\cX}{\mathcal{X}}

\newcommand{\Hom}{\mathrm{Hom}}

\newcommand{{\inv} }{\mathrm{inv}}
\newcommand{\ev}{\mathrm{ev}}
\newcommand{\Aut}{\mathrm{Aut}}

\newcommand{\val}{ {\mathrm{val}} }
\newcommand{\vir}{{\mathrm{vir}}}

\newcommand{\NE}{ {\mathrm{NE}}}

\newcommand{\bt}{\mathbf{t}}

\newcommand{\su}{\mathsf{u}}
\newcommand{\sv}{\mathsf{v}}
\newcommand{\sw}{\mathsf{w}}

\newcommand{\Mbar}{\overline{\cM}}

\newcommand{\CP}{\bP^1}
\newcommand{\RP}{\bR\bP^1}
\newcommand{\OGW}{(\CP,L),S^1}

\newcommand{\pOne}{{p_{\si_1}}}
\newcommand{\pZero}{{p_{\si_0}}}
\newcommand{\pTwo}{{p_{\si_2}}}

\newtheorem{lma}{Lemma}[section]

\newtheorem{defn}[lma]{Definition}
\newtheorem{prop}[lma]{Proposition}

\newtheorem{theorem}[lma]{Theorem}
\newtheorem{remark}[lma]{Remark}

\theoremstyle{definition}

\begin{document}
\allowdisplaybreaks
\title{Mirror symmetry and open/closed correspondence for the projective line}

\author{Jinghao Yu}
\address{Jinghao Yu, Department of Mathematical Sciences, Tsinghua University, Haidian District, Beijing 100084, China}
\email{yjh21@mails.tsinghua.edu.cn}

\author{Zhengyu Zong}
\address{Zhengyu Zong, Department of Mathematical Sciences,
	Tsinghua University, Haidian District, Beijing 100084, China}
\email{zyzong@mail.tsinghua.edu.cn}

\maketitle

\begin{abstract}
	We study the open/closed correspondence for the projective line via mirror symmetry. More explicitly, we establish a correspondence between the generating function of disk Gromov-Witten invariants of the complex projective line $\bP^1$ with boundary condition specified by an $S^1$-invariant Lagrangian sub-manifold $L$ and the asymptotic expansion of the $I$-function of a toric surface $\cS$.
\end{abstract}

\tableofcontents

\section{Introduction}

\subsection{Historical background and motivation}

\subsubsection{Open/closed correspondence for Calabi-Yau 3-folds}
Proposed by Mayr \cite{Mayr01} and Lerche-Mayr \cite{LM01}, the \emph{open/closed correspondence} predicts that the genus-zero topological amplitudes of an open string geometry on a
Calabi-Yau 3-fold with a prescribed Lagrangian boundary condition should coincide with those of a closed string geometry on a dual Calabi-Yau 4-fold. In mathematical language, the open/closed correspondence conjecturally relates the disk Gromov-Witten invariants of the open 3-fold geometry to the genus-zero closed Gromov-Witten invariants of the 4-fold geometry.

The open/closed correspondence for the case of a toric Calabi-Yau 3-fold $X$ with a Lagrangian submanifold $L$ of Aganagic-Vafa type is mathematically proved in \cite{LY21} by virtual localization techniques. The above result is generalized to the case of a toric Calabi-Yau 3-orbifold $\cX$ with a Lagrangian suborbifold $\cL$ of Aganagic-Vafa type in \cite{LY22}. In \cite{AL23}, the open/closed correspondence is also proved for the quintic threefold in terms of Gauged Linear Sigma Model. By the open/relative correspondence for toric Calabi-Yau 3-orbifolds in \cite{FLT12}, the open/closed correspondence for toric Calabi-Yau 3-orbifolds can also be viewed as the log-local correspondence \cite{vGGR19}. Related works can be found in e.g. \cite{BBvG20,BBvG20b}.

\subsubsection{Mirror symmetry and open/closed correspondence for the projective line}
In this paper, we prove the open/closed correspondence for the complex projective line $\bP^1$ via mirror symmetry, although $\bP^1$ is not Calabi-Yau. 

Let $t\in S^1$ act on $\CP$ by $t\cdot [z_1,z_2] = [tz_1,  t^{-1}z_2]$, where $[z_1, z_2]$ are the homogeneous coordinates of $\bP^1$. Let $L := \{[e^{{\rm i}\varphi},e^{-{\rm i}\varphi}]\in\CP: \varphi\in\bR\}$ be the Lagrangian submanifold of $\CP$, which is preserved by the $S^1$-action. By taking a M\"{o}bius transform, we can identify the pair $(\bP^1,L)$ with $(\bP^1,\bR\bP^1)$. In Section \ref{sec:open-gw-p1}, we will define and study the $S^1$-equivariant open Gromov-Witten theory of $(\bP^1,L)$. The open Gromov-Witten theory with descendants of $(\bP^1,\bR\bP^1)$ is studied in \cite{BNPT22}. Related works can be found in \cite{BT17,Net17,PST14,Tes23}.

On the other hand, we will define a toric surface $\cS$ in Section \ref{sec:geometryX} and study the equivariant closed Gromov-Witten theory of $\cS$ in Section \ref{sec:closed-gw-X}. We will consider the $J$-function $J_\cS(\btau,z)$, which encodes the genus zero Gromov-Witten invariants of $\cS$. By genus zero mirror theorem, the $J$-function $J_\cS(\btau,z)$ is identified to the $I$-function $I_\cS(\bq,z)$. The main result (Theorem \ref{thm:main-thm-i}) of this paper states that the generating function of the $S^1$-equivariant open Gromov-Witten invariants of $(\bP^1,L)$ can be identified to the coefficient of the $z^{-2}$-term in the asymptotic expansion of $I_\cS(\bq,z)$.

In \cite{Z25}, the second author studies the open/closed correspondence for $(\bP^1,L)$ via virtual localization computations. We would like to remark the following differences between the current paper and \cite{Z25}. In \cite{Z25}, the \emph{descendant} insertions are included in both open Gromov-Witten invariants of $(\bP^1,L)$ and closed Gromov-Witten invariants of $\cS$ while in the current paper we only consider primary insertions. On the other hand, the advantage of the current paper is that the main result (Theorem \ref{thm:main-thm-i}) takes a more elegant form. Besides, the study of open/closed correspondence in \cite{Z25} is at numerical level and is purely on A-model side. In the current paper, the correspondence is studied via mirror symmetry and is upgraded to the level of generating
functions. Therefore the correspondence further carries over to the B-model side, predicting that the B-model disk potential $W_{0,1}$ (studied in \cite{YZ25} via mirror curve) and the $I$-function $I_\cS$ match up.

\begin{figure}
    \begin{equation*}\label{fig:correspondence}
	    \xymatrix{
			F_{0,1}^{\OGW} \ar[r]^-{\textrm{mirror }} & W_{0,1}  \\
		    J_\cS \ar[r]^-{\textrm{mirror }} \ar[u]  &  I_\cS \ar[u]
	    }
    \end{equation*}
\caption{Interrelations among the mentioned topics}
\end{figure}

We hope the result in this paper can contribute to understanding of the open/closed correspondence for non-Calabi-Yau target spaces.

\subsubsection{Construction of toric Calabi-Yau manifold/orbifolds in open/closed correspondence}
From geometric point of view, the formalism of open/closed correspondence in this paper has a similar philosophy as in the case of toric Calabi-Yau 3-folds/3-orbifolds.

We sketch the geometric construction in the study of open/closed correspondence for toric Calabi-Yau 3-folds/3-orbifolds following \cite{LY21,LY22}.
For the open sector, \cite{LY21,LY22} consider the toric Calabi-Yau 3-fold $\cX$ defined by an extended stacky fan $\bSi=(N,\Si,\alpha)$ in the sense of Jiang \cite{J08}, 
where $N\cong \bZ^3$, $\Si$ is a finite simplicial fan in $\bR^3=\bZ^3\otimes \bR$ and $\alpha:\bZ^R\rightarrow N$. 
The corresponding dual toric Calabi-Yau 4-fold $\widetilde{\cX}$ is specified by the extended stacky fan $\widetilde{\bSi}=(\widetilde{N},\widetilde{\Si},\widetilde{\alpha})$,
where $\widetilde{N}=N\oplus \bZ v_4$, $\widetilde{\alpha}:\bZ^{R+2}\rightarrow \widetilde{N}$ is an extension of morphism $\alpha$.
The lattice map fits into the following commutative diagram:
\begin{equation}\label{eqn:ses}
\begin{tikzcd}
0 \arrow[r] & \bL \arrow[r] \arrow[d] & \bZ^R \arrow[r, "\alpha"] \arrow[d] & N \arrow[r] \arrow[d] & 0 \\
0 \arrow[r] & \widetilde{\bL} \arrow[r]    & \bZ^{R+2} \arrow[r, "\widetilde{\alpha}"]  & \widetilde{N} \arrow[r]     & 0,
\end{tikzcd}
\end{equation}
where $\bL:=\ker(\alpha)\cong \bZ^{R-3}$, $\widetilde{\bL}:= \ker(\widetilde{\alpha})\cong\bZ^{R-2}$. 

In Section \ref{sec:comparison}, we construct the toric surface $\cS$ via a similar philosophy. One can compare the two commutative diagrams \eqref{eqn:ses} and \eqref{eqn:diagram}. We hope that similar constructions can be applied to more general spaces, especially to non-Calabi-Yau manifold/orbifolds.

\subsubsection{Relation to open Gromov-Witten invariants via coherent boundary conditions}
We would like to briefly remark the relationship between the open Gromov-Witten invariants studied in this paper and the open Gromov-Witten invariants via coherent boundary conditions studied in \cite{BNPT22}. In Section \ref{sec:disk}, we define the open Gromov-Witten invariants of $(\bP^1,\bR\bP^1)$ via integration over the fixed locus of the $S^1$-action. In \cite{BNPT22}, the open Gromov-Witten invariants of $(\bP^1,\bR\bP^1)$ is defined via coherent boundary conditions and a graph sum formula is obtained. The vertex factor in this graph sum formula is given by the open Gromov-Witten invariants via integration over the fixed locus defined in Section \ref{sec:disk}, playing the role of the building block of the open Gromov-Witten invariants via coherent boundary conditions.

In the study of the open/closed correspondence for the case of toric Calabi-Yau 3-folds/3-orbifolds \cite{LY21,LY22}, the open Gromov-Witten invariants are also defined via integration over the fixed locus. The formalism of open/closed correspondence in our case has a similar philosophy and takes an elegant form (Theorem \ref{thm:main}). If we want to formulate the open/closed correspondence using open Gromov-Witten invariants via coherent boundary conditions, we can apply the graph sum formula to the closed Gromov-Witten invariants of $\cS$ and modify the statement of Theorem \ref{thm:main} accordingly.

\subsection{Statement of the main result}
Let $\bP^1$ be the complex projective line with homogeneous coordinates $[z_1, z_2]$. Consider the $S^1$ action on $\bP^1$ defined as
$$
 t\cdot [z_1,z_2] = [tz_1,  t^{-1}z_2],
$$
where $t\in S^1$. Let $\bC[\sv] = H^*_{S^1}(\text{point};\bC)$ be the $S^1$-equivariant cohomology of a point.
The $S^1$-equivariant cohomology of $\bP^1$ is given by
\[
H^*_{S^1}(\bP^1;\bC) = \bC[H,\sv]/\langle (H+\sv/2)(H-\sv/2)\rangle,
\]
where $\deg H=\deg \sv=2$.

Let 
$$   
L := \{[e^{{\rm i}\varphi},e^{-{\rm i}\varphi}]\in\CP: \varphi\in\bR\}
$$
be the Lagrangian submanifold of $\CP$, which is preserved by the $S^1$-action. By taking a M\"obius transform, we can identify the pair $(\CP,L)$ with $(\CP,\RP)$. We have $H_1(L)\cong\bZ$.

In Section \ref{sec:open-gw-p1}, we will study the disk Gromov-Witten invariants of $(\CP,L)$, which count holomorphic maps from the disk to $(\CP,L)$. We will consider the generating function $F_{0,1}^{\OGW}(\bt;X)$ of disk Gromov-Witten invariants of $(\CP,L)$, where $\bt = t^01+t^1H\in H^*_{S^1}(\bP^1;\bC)$ and $X$ is a formal variable encoding the winding number.

In Section \ref{sec:geometryX}, we will define a toric surface constructed as follows. Let $N=\bZ^2$ and define $v_1, v_2, v_3, v_4\in N$ as 
\[
v_1 = (0,1), \quad v_2 = (1,0), \quad v_3 = (-1,1),\quad v_4=(1,-1).
\]
Define 2-dimensional cones $\si_0,\si_1,\si_2\subset N_\bR$ as 
\[
\si_0 = \bR_{\geq 0}v_1 + \bR_{\geq 0}v_2, \quad \si_1 = \bR_{\geq 0}v_1 + \bR_{\geq 0}v_3, 
\quad \si_2 = \bR_{\geq 0}v_2 + \bR_{\geq 0}v_4.
\]
Let $\Si$ be the fan with top dimensional cones $\si_0,\si_1,\si_2$ and let $\cS$ be the toric surface defined by $\Si$ (see Figure \ref{fig:toric_surface}). The torus $T := N\otimes \bC^* \cong (\bC^*)^2$ acts on $\cS$ canonically.

In Section \ref{sec:closed-gw-X}, we will study the $T$-equivariant closed Gromov-Witten invariants of $\cS$. In particular, we will consider the $T$-equivariant $J$-function $J_\cS(\btau,z)$, which encodes the genus zero $T$-equivariant Gromov-Witten invariants of $\cS$. Here $\btau\in H^*_T(\cS)$ and $z$ is a formal variable encoding the descendant insertion (See Section \ref{sec:JX}). By genus zero mirror theorem, the $J$-function $J_\cS(\btau,z)$ is identified to the $I$-function $I_\cS(\bq,z)$, which is an explicit generalized hypergeometric series (See Section \ref{sec:IX}).

The following theorem is the main result of this paper:
\begin{theorem}[=Theorem \ref{thm:main-thm-i}]\label{thm:main}
	Under the relation $\log q_0 = t^0$, $q_1 = -\sqrt{q}X^{-1}$ and $q_2 = -\sqrt{q}X$,
	we have
	\[
	F_{0,1}^{\OGW} (\bt;X) = [z^{-2}]\big(I_\cS(\bq,z), \su_1\widetilde{\phi}_0\big)_{\cS,T}\Big|_{\su_2=-\su_1=\sv} + \text{Exc},
	\]
	where the $I$-function is in the asymptotic expansion as $\sv \rightarrow \infty$, and the exceptional term is 
	$\text{Exc} := -\sqrt{q}X^{-1} + \sqrt{q}X- \frac{(t^0)^2}{2\sv} - q\sv^{-1}$.
\end{theorem}
Another way to understand the right hand side of Theorem \ref{thm:main} is given in Section \ref{sec:formal expansion} from formal point of view.

\begin{figure}[!ht]
    \centering
    \begin{tikzcd}
        {F_{0,1}^{(\mathbb{P}^1,L),S^1}} \arrow[dd, dashed, <->] &  &  &                                            &   {(\mathbb{P}^1,L)} \text{\ open} \\
                                                            &  &  &                                            &                          \\
        {J_\cS({\btau},z)}  &  &  & {I_\cS(\bq,z)} \arrow[llluu, "\text{Theorem }\ref{thm:main-thm-i}"']\arrow[lll, "\text{mirror symmetry}"] &  \cS \text{\ closed}                \\
        \text{A-model}                                             &  &  & \text{B-model}                                    &  &                        
    \end{tikzcd}
    \caption{open/closed correspondence and mirror symmetry}
\end{figure}

\subsection{Overview of the paper}
In Section \ref{sec:geometric-setup}, we review the open geometry of $(\CP,L)$ and the closed geometry of the toric surface $\cS$.
In Section \ref{sec:open-gw-p1}, we review the open $S^1$-equivariant Gromov-Witten theory of $(\CP,L)$ and give an explicit formula for the disk potential.
In Section \ref{sec:closed-gw-X}, we study the equivariant closed Gromov-Witten theory of $\cS$. We will study the $J$-function of $\cS$ and identify it to the $I$-function by genus zero mirror theorem.
In Section \ref{sec:open-closed}, we study the correspondence between the disk potential of $(\CP,L)$ and the $I$-function of $\cS$, which is the main theorem of this paper.

\subsection*{Acknowledgements}
The authors would like to thank Song Yu for helpful explanations on the relationship between mirror symmetry and open/closed correspondence. The authors would also like to thank Bohan Fang and Chiu-Chu Melissa Liu for useful discussions. The second author is partially supported by the Natural Science Foundation of Beijing, China (grant No. 1252008) and NSFC (grant No. 12571067).

\section{Geometric setup}\label{sec:geometric-setup}

\subsection{Equivariant cohomology of $\CP$}
Let $t\in S^1$ act on $\CP$ by
\begin{equation*}\label{eq:S1-translation}
     t\cdot [z_1,z_2] = [tz_1,  t^{-1}z_2].
\end{equation*}
Let $\bC[\sv] = H^*_{S^1}(\text{point};\bC)$ be the $S^1$-equivariant cohomology of a point.
The $S^1$-equivariant cohomology of $\bP^1$ is given by
\[
    H^*_{S^1}(\bP^1;\bC) = \bC[H,\sv]/\langle (H+\sv/2)(H-\sv/2)\rangle.
\]
Let $p_1=[1,0]$ and $p_2=[0,1]$ be the $S^1$-fixed points. Then $H|_{p_1}=-\sv/2$, $H|_{p_2}=\sv/2$. 
The $S^1$-equivariant Poincaré dual of $p_1$ and $p_2$ are $H-\sv/2$ and $H+\sv/2$, respectively.

Let 
\[
    \phi_1 := -\frac{H-\sv/2}{\sv}, \phi_2:= \frac{H+\sv/2}{\sv} \in H^*_{S^1}(\bP^1;\bC)\otimes_{\bC[\sv]}\bC(\sv).
\]
We have
\[
    \phi_\alpha\cup\phi_\beta = \delta_{\alpha\beta}\phi_{\alpha}, \quad \alpha,\beta = 1,2.
\]

Let 
$$   
    L := \{[e^{{\rm i}\varphi},e^{-{\rm i}\varphi}]\in\CP: \varphi\in\bR\}
$$
be the Lagrangian submanifold of $\CP$, which is preserved by the $S^1$-action. By taking a M\"obius transform, we can identify the pair $(\CP,L)$ with $(\CP,\RP)$. Let $D_1$ and $D_2$ be the two disks with boundary $L$ centered at $p_1$ and $p_2$ respectively.
Then we have
\[
    H_2(\bP^1,L) = \bZ[D_1]\oplus \bZ[D_2].
\]
We identify the relative homology group $H_2(\bP^1,L)$ to $\bZ^2$,
where $\beta'=(d_-,d_+)\in \bZ^2$ is identified to $d_-[D_1]+d_+[D_2]$. 
Let $E(\CP,L) = \bZ^2_{\geq 0}$ be the set of effective curve classes of $H_2(\CP,L)$.

On the other hand, the homology group $H_2(\bP^1,L)$ satisfies the short exact sequence
    \begin{equation*}%\label{eqn:P1-L}
        \begin{tikzcd}
            0 \arrow[r] & H_2(\bP^1)= \bZ[\bP^1] \arrow[r, "\delta"] & H_2(\bP^1,L) \arrow[r, "\partial"] & H_1(L)= \bZ[\RP] \arrow[r] & 0.
        \end{tikzcd}
    \end{equation*}
For $d[\bP^1]\in H_2(\bP^1)$, the map $\delta:H_2(\bP^1)\rightarrow H_2(\bP^1,L)$ is defined by $\delta(d[\bP^1])=(d,d)$.
The connecting map $\partial:H_2(\bP^1,L)\rightarrow H_1(L)$ is given by $\partial(a,b) = (b-a)[\RP]$.

\subsection{The geometry of toric surface $\cS$}\label{sec:geometryX}
In this subsection, we construct a toric surface $\cS$ and study its geometry. We refer to \cite{CLS11,Fulton93} for the general notations of
toric varieties.

Let $N=\bZ^2$ and define $v_1, v_2, v_3, v_4\in N$ as 
\[
    v_1 = (0,1), \quad v_2 = (1,0), \quad v_3 = (-1,1),\quad v_4=(1,-1).
\]
Let $\tau_i = \bR_{\geq 0}v_i\subset N_{\bR} := N\otimes \bR, i = 1,2,3,4$ be the corresponding 1-dimensional cones.
Define 2-dimensional cones $\si_0,\si_1,\si_2\subset N_\bR$ as 
\[
    \si_0 = \bR_{\geq 0}v_1 + \bR_{\geq 0}v_2, \quad \si_1 = \bR_{\geq 0}v_1 + \bR_{\geq 0}v_3, 
    \quad \si_2 = \bR_{\geq 0}v_2 + \bR_{\geq 0}v_4.
\]
Let $\Si$ be the fan with top dimensional cones $\si_0,\si_1,\si_2$ and let $\cS$ be the toric surface defined by $\Si$ (see Figure \ref{fig:toric_surface}).

The torus $T := N\otimes \bC^* \cong (\bC^*)^2$ acts on $\cS$. Let $p_{\si_i} = V(\si_i), \ i=0,1,2$ be the $T$-fixed points and let
$l_{\tau_i} = V(\tau_i), \ i = 1,2,3,4$ be the $T$-invariant lines. 
Let $M:= \Hom(N,\bZ) = \Hom(T,\bC^*)$ be the character lattice of $T$. For $\tau_i\subset \si_j$, let
$\sw(\tau_i,\si_j)$ be the weight of the $T$-action on $T_{p_{\si_j}}l_{\tau_i}$, the tangent line to $l_{\tau_i}$ at the fixed point $p_{\si_j}$. The
weights $\sw(\tau_i,\si_j)$ are given by 
\begin{equation*}
    \begin{array}{rll}
        \sw(\tau_1,\si_1) = \su_1, & \sw(\tau_1,\si_0) = -\su_1, & \sw(\tau_2,\si_2) = \su_2,
        \\[1.1ex]
        \sw(\tau_3,\si_1) = -\su_1-\su_2, & \sw(\tau_2,\si_0)=-\su_2, & \sw(\tau_4,\si_2) = -\su_1-\su_2.
    \end{array}
\end{equation*}

\begin{figure}[h]
\center

\tikzset{every picture/.style={line width=0.65pt}} %set default line width to 0.75pt        

\begin{tikzpicture}[x=0.75pt,y=0.75pt,yscale=-1,xscale=1]
%uncomment if require: \path (0,326); %set diagram left start at 0, and has height of 326

%Straight Lines [id:da3694030533027516] 
\draw    (160,160) -- (208.59,208.59) ;
\draw [shift={(210,210)}, rotate = 225] [color={rgb, 255:red, 0; green, 0; blue, 0 }  ][line width=0.75]    (10.93,-3.29) .. controls (6.95,-1.4) and (3.31,-0.3) .. (0,0) .. controls (3.31,0.3) and (6.95,1.4) .. (10.93,3.29)   ;
%Straight Lines [id:da1291822452404594] 
\draw    (160,160) -- (208,160) ;
\draw [shift={(210,160)}, rotate = 180] [color={rgb, 255:red, 0; green, 0; blue, 0 }  ][line width=0.75]    (10.93,-3.29) .. controls (6.95,-1.4) and (3.31,-0.3) .. (0,0) .. controls (3.31,0.3) and (6.95,1.4) .. (10.93,3.29)   ;
%Straight Lines [id:da8384310037117084] 
\draw    (160,160) -- (160,112) ;
\draw [shift={(160,110)}, rotate = 90] [color={rgb, 255:red, 0; green, 0; blue, 0 }  ][line width=0.75]    (10.93,-3.29) .. controls (6.95,-1.4) and (3.31,-0.3) .. (0,0) .. controls (3.31,0.3) and (6.95,1.4) .. (10.93,3.29)   ;
%Straight Lines [id:da0977029027289894] 
\draw    (160,160) -- (111.41,111.41) ;
\draw [shift={(110,110)}, rotate = 45] [color={rgb, 255:red, 0; green, 0; blue, 0 }  ][line width=0.75]    (10.93,-3.29) .. controls (6.95,-1.4) and (3.31,-0.3) .. (0,0) .. controls (3.31,0.3) and (6.95,1.4) .. (10.93,3.29)   ;
%Straight Lines [id:da9843802660981886] 
\draw    (371.56,99.33) -- (328.78,139.91) ;
\draw [shift={(327.33,141.29)}, rotate = 316.51] [color={rgb, 255:red, 0; green, 0; blue, 0 }  ][line width=0.75]    (10.93,-3.29) .. controls (6.95,-1.4) and (3.31,-0.3) .. (0,0) .. controls (3.31,0.3) and (6.95,1.4) .. (10.93,3.29)   ;
%Straight Lines [id:da913798373067308] 
\draw    (371.56,99.33) -- (470,99.33) ;
%Straight Lines [id:da9774468413973649] 
\draw    (470,99.33) -- (470,183.24) ;
%Straight Lines [id:da15922838461767685] 
\draw    (470,183.24) -- (427.23,223.82) ;
\draw [shift={(425.78,225.2)}, rotate = 316.51] [color={rgb, 255:red, 0; green, 0; blue, 0 }  ][line width=0.75]    (10.93,-3.29) .. controls (6.95,-1.4) and (3.31,-0.3) .. (0,0) .. controls (3.31,0.3) and (6.95,1.4) .. (10.93,3.29)   ;
%Straight Lines [id:da5136861427585443] 
\draw    (371.56,99.33) -- (404.6,99.3) ;
\draw [shift={(406.6,99.3)}, rotate = 179.95] [color={rgb, 255:red, 0; green, 0; blue, 0 }  ][line width=0.75]    (10.93,-3.29) .. controls (6.95,-1.4) and (3.31,-0.3) .. (0,0) .. controls (3.31,0.3) and (6.95,1.4) .. (10.93,3.29)   ;
%Straight Lines [id:da2936049618196055] 
\draw    (470,99.33) -- (437,99.3) ;
\draw [shift={(435,99.3)}, rotate = 0.05] [color={rgb, 255:red, 0; green, 0; blue, 0 }  ][line width=0.75]    (10.93,-3.29) .. controls (6.95,-1.4) and (3.31,-0.3) .. (0,0) .. controls (3.31,0.3) and (6.95,1.4) .. (10.93,3.29)   ;
%Straight Lines [id:da47715145926211056] 
\draw    (470,99.33) -- (470.12,125.3) ;
\draw [shift={(470.13,127.3)}, rotate = 269.73] [color={rgb, 255:red, 0; green, 0; blue, 0 }  ][line width=0.75]    (10.93,-3.29) .. controls (6.95,-1.4) and (3.31,-0.3) .. (0,0) .. controls (3.31,0.3) and (6.95,1.4) .. (10.93,3.29)   ;
%Straight Lines [id:da9299959587062113] 
\draw    (470,183.24) -- (470,157.2) ;
\draw [shift={(470,155.2)}, rotate = 90] [color={rgb, 255:red, 0; green, 0; blue, 0 }  ][line width=0.75]    (10.93,-3.29) .. controls (6.95,-1.4) and (3.31,-0.3) .. (0,0) .. controls (3.31,0.3) and (6.95,1.4) .. (10.93,3.29)   ;

% Text Node
\draw (153,98.87) node [anchor=north west][inner sep=0.75pt]    {$v_{1}$};
% Text Node
\draw (215.33,154.87) node [anchor=north west][inner sep=0.75pt]    {$v_{2}$};
% Text Node
\draw (87.33,98.87) node [anchor=north west][inner sep=0.75pt]    {$v_{3}$};
% Text Node
\draw (215.33,202.2) node [anchor=north west][inner sep=0.75pt]    {$v_{4}$};
% Text Node
\draw (181.33,122.87) node [anchor=north west][inner sep=0.75pt]    {$\sigma_{0}$};
% Text Node
\draw (137,120.87) node [anchor=north west][inner sep=0.75pt]    {$\sigma_{1}$};
% Text Node
\draw (185.33,171.4) node [anchor=north west][inner sep=0.75pt]    {$\sigma_{2}$};
% Text Node
\draw (473.2,82.13) node [anchor=north west][inner sep=0.75pt]    {$p_{\sigma_{0}}$};
% Text Node
\draw (465.2,95.73) node [anchor=north west][inner sep=0.75pt]    {$\bullet$};
% Text Node
\draw (473.2,176.8) node [anchor=north west][inner sep=0.75pt]    {$p_{\sigma_{2}}$};
% Text Node
\draw (465.2,178.97) node [anchor=north west][inner sep=0.75pt]    {$\bullet $};
% Text Node
\draw (359.53,82.13) node [anchor=north west][inner sep=0.75pt]    {$p_{\sigma_{1}}$};
% Text Node
\draw (365.9,95.73) node [anchor=north west][inner sep=0.75pt]    {$\bullet$};
% Text Node
\draw (411.6,82.13) node [anchor=north west][inner sep=0.75pt]    {$l_{\tau_{1}}$};
% Text Node
\draw (473.2,134.8) node [anchor=north west][inner sep=0.75pt]    {$l_{\tau_{2}}$};
% Text Node
\draw (387.08,103.73) node [anchor=north west][inner sep=0.75pt]    {$\su_{1}$};
% Text Node
\draw (422.78,102.73) node [anchor=north west][inner sep=0.75pt]    {$-\su_{1}$};
% Text Node
\draw (440.53,114.73) node [anchor=north west][inner sep=0.75pt]    {$-\su_{2}$};
% Text Node
\draw (450.2,157.07) node [anchor=north west][inner sep=0.75pt]    {$\su_{2}$};
% Text Node
\draw (299.2,145.73) node [anchor=north west][inner sep=0.75pt]    {$-\su_{1}-\su_{2}$};
% Text Node
\draw (358.87,211.07) node [anchor=north west][inner sep=0.75pt]    {$-\su_{1}-\su_{2}$};
% Text Node
\draw (329.2,103.8) node [anchor=north west][inner sep=0.75pt]    {$l_{\tau _{3}}$};
% Text Node
\draw (448.53,203.8) node [anchor=north west][inner sep=0.75pt]    {$l_{\tau _{4}}$};

\end{tikzpicture}
\caption{The fan of $\Sigma$ and 1-skeleton of $\cS$}
\label{fig:toric_surface}
\end{figure}

Let
\[
    \widetilde{\phi}_1 := \frac{[p_{\si_1}]}{-\su_1-\su_2},\quad
    \widetilde{\phi}_2 := \frac{[p_{\si_2}]}{-\su_1-\su_2},\quad
    \widetilde{\phi}_0 := \frac{[p_{\si_0}]}{\su_1\su_2}.
\]
$\{\widetilde{\phi}_i: i= 0,1,2\}$ is a basis of $H^*_T(\cS;\bC)\otimes_{\bC[\su_1,\su_2]}\bC(\su_1,\su_2)$.
We have the homology group $H_2(\cS;\bZ) = \bZ l_{\tau_1}\oplus \bZ l_{\tau_2}$. So we make the identification $H_2(\cS;\bZ)\cong \bZ^2$, where 
$(d_1,d_2)\in\bZ^2$ is identified to $d_1l_{\tau_1}+d_2l_{\tau_2}$. Let $\NE(\cS)\subset H_2(\cS;\bR)$ be the Mori cone 
generated by effective curve classes in $\cS$, and $E(\cS) \cong \bZ^2_{\geq 0}$ denote the semigroup $\NE(\cS)\cap H_2(\cS;\bZ)$.

Consider the homomorphism
\[
    \phi: \tilde{N}:=\bigoplus_{i=1}^4 \bZ \tilde{v}_i\rightarrow N, \quad \tilde{v}_i\mapsto v_i.
\]
Let $\bL =\ker(\phi)\cong \bZ^2$, then we have a short exact sequence of abelian groups
\[
    0\rightarrow \mathbb{L}\xrightarrow{\psi} \mathbb{Z}^{4} \xrightarrow{\phi} \mathbb{Z}^2 \rightarrow 0.
\]
Let $e_1,e_2$ be the basis of $\bL$ such that under the basis of $\bL, \ \tilde{N}$ and $N$, we have
\begin{equation*}
    \phi = \left[
        \begin{array}{cccc}
            0 & 1 & -1 & 1
            \\
            1 & 0 & 1 & -1
        \end{array}
    \right],
\quad
    \psi = \left[
        \begin{array}{cc}
            -1 & 1
            \\
            1 & -1
            \\
            1 & 0
            \\
            0 & 1
        \end{array}
    \right].
\end{equation*}
Let $\{e_1^\vee,e_2^\vee\}$ be the dual $\bZ$-basis of $\bL^\vee$, and define $D_i\in\bL^\vee, i = 1,2,3,4$ as row vectors of $\psi$:
\[
    D_1= (-1,1), \quad D_2 = (1,-1), \quad D_3 = (1,0),\quad D_4=(0,1).
\]
There is a canonical identification $\bL^\vee\cong H^2(\cS;\bZ)$, where the divisor classes $D_i$ is identified to 
\[
     [V(\tau_i)] = [l_{\tau_i}] \in H^2(\cS;\bZ).
\]
The nef cone of $\cS$ is 
\[
    \text{Nef}(\cS) = \sum_{i=3,4}\bR_{\geq 0} D_i.  
\]

Let $H_1^T, H_2^T\in H^2_T(\cS)$ be the $T$-equivariant lift of Poincaré dual of $l_{\tau_3}, l_{\tau_4}$
satisfying:
\begin{equation*}
    \begin{array}{lll}
        H_1^T|_{\pOne} = \su_1, &H_1^T|_{\pZero}=0,&H_1^T|_\pTwo = 0,
        \\[1.2ex]
        H_2^T|_\pOne = 0, &  H_2^T|_{\pZero} = 0,& H_2^T|_\pTwo = \su_2.
    \end{array}
\end{equation*}
We define the $T$-equivariant divisor classes $D_i^T:= [V(v_i)]\in H^2_T(\cS)$
\begin{equation*}
    \begin{aligned}
        D_1^T &:= -H_1^T + H_2^T - \su_2,
        \\
        D_2^T &:= H_1^T - H_2^T -\su_1,
        \\
        D_3^T &:= H_1^T ,
        \\
        D_4^T &:= H_2^T.  
    \end{aligned}
\end{equation*}
We have
\begin{equation*}
    \begin{array}{lll}
        D_1^T|_\pOne = -\su_1-\su_2,& D_1^T|_\pZero = -\su_2, & D_1^T|_\pTwo = 0,
        \\[1.2ex]
        D_2^T|_\pOne = 0,& D_2^T|_\pZero = -\su_1, & D_2^T|_\pTwo = -\su_1-\su_2,
        \\[1.2ex]
        D_3^T|_{\pOne} = \su_1, & D_3^T|_{\pZero}=0,& D_3^T|_\pTwo = 0,
        \\[1.2ex]
        D_4^T|_\pOne = 0, &  D_4^T|_{\pZero} = 0, & D_4^T|_\pTwo = \su_2.
    \end{array}
\end{equation*}

Under the identification $e_1\mapsto l_{\tau_1}, \ e_2\mapsto l_{\tau_2}$,
the effective curve class $E(\cS) = \{\beta \in \bL: \beta = d_1e_1+d_2e_2,\ d_1,d_2\geq 0\}$.

\subsection{Comparison of geometry}\label{sec:comparison}
The projective line $\bP^1$ is constructed from a fan in $\bZ$. Let us consider the homomorphism 
\[
\alpha: \bZ \tilde{b}_1 \oplus \bZ \tilde{b}_2 \to \bZ b
\]
defined by $\alpha(\tilde{b}_1) = -b$ and $\alpha(\tilde{b}_2) = b$. This homomorphism is surjective and fits into the following short exact sequence of lattices:
\[
0 \to \bK \to \mathbb{Z}^2 \xrightarrow{\alpha} \mathbb{Z} \to 0,
\]
where $\bK := \ker(\alpha) = \bZ\epsilon$. 

The short exact sequences of lattices associated with $\bP^1$ and $\mathcal{S}$ fit into the commutative diagram below:
\begin{equation}\label{eqn:diagram}
\begin{tikzcd}
0 \arrow[r] & \bK \arrow[r] \arrow[d, "\kappa"] & \bZ^2 \arrow[r, "\alpha"] \arrow[d, "\hat{\nu}"] & \bZ \arrow[r] \arrow[d, "\nu"] & 0 \\
0 \arrow[r] & \bL \arrow[r, "\psi"]             & \bZ^4 \arrow[r, "\phi"]                    & \bZ^2 \arrow[r]                   & 0
\end{tikzcd}
\end{equation}
with the mappings specified by $\kappa(\epsilon) = e_1 + e_2$, $\nu(b) = v_4$, and $\hat{\nu}(\tilde{b}_i) = \tilde{v}_{i+2}$ for $i = 1,2$.

Take the dual map $\Hom(-,\bZ)$ to the commutative diagram \eqref{eqn:diagram}, we get 
\begin{equation}\label{eqn:pullback}
    \iota^*: H^2_T(\cS;\bZ)\rightarrow H^2_{S^1}(\bP^1;\bZ),
\end{equation}
where $\iota^*(D_1^T) = -\sv$, $\iota^*(D_2^T) = \sv$, $\iota^*(D_3^T) = \iota^*(D_4^T) = H$.

\begin{remark}
  The geometric construction of $\mathcal{S}$ is motivated by the goal of associating the Lagrangian submanifold $L \subset \CP$ with a new cone $\sigma_0$ in $\Sigma$. 
This aligns with the construction philosophy in \cite{LY21,LY22}, where the Lagrangian submanifold $\mathcal{L}$ is associated with a 4-dimensional cone $\widetilde{\sigma}_0$ defined by the set $\{\widetilde{b}_i: i=2,3,R+1,R+2\}$.
\end{remark}

% The inclusion from the fan of $\CP$ to $\Si$ (the fan of $\cS$) induces an inclusion $\iota: \CP\rightarrow \cS$.
% The homomorphism $\iota_*: H_2(\CP,L;\bZ)\rightarrow H_2(\cS;\bZ)$ is given by identifying $(d_-,d_+)$ and $(d_1,d_2)$.

\section{Gromov-Witten theory of $\bP^1$}\label{sec:open-gw-p1}
\subsection{Equivariant Gromov-Witten invariants of $\CP$}
Let $E(\bP^1)$ denote the set of effective curve classes in $H_2(\bP^1;\bZ)$.
Given a nonnegative integer $n$ and an effective curve class $\beta\in E(\bP^1)$,
let $\overline{\cM}_{0,n}(\bP^1,\beta)$ be the moduli stack of genus-0, $n$-pointed, degree-$\beta$ stable maps to $\bP^1$.
Let $\ev_i: \overline{\cM}_{0,n}(\bP^1,\beta)\rightarrow \bP^1$ be the evaluation map at the $i$-th marked point. The $S^1$-action on $\bP^1$ induces an $S^1$-action 
on $\overline{\cM}_{0,n}(\bP^1,\beta)$ and the evaluation map $\ev_i$ is $S^1$-equivariant.

For $i=1,\dots,n$, let $\bL_i$ be the $i$-th tautological line bundle over $\overline{\cM}_{0,n}(\bP^1,\beta)$ formed by the cotangent line at the $i$-th marked point.
Define the $i$-th descendant class $\psi_i$ as 
$$
    \psi_i := c_1(\bL_i)\in H^2(\overline{\cM}_{0,n}(\bP^1,\beta);\bQ).
$$
We choose an $S^1$-equivariant lift $\psi_i^{S^1}\in H^2_{S^1}(\overline{\cM}_{0,n}(\bP^1,\beta);\bQ)$ of $\psi_i$.

Given $\gamma_1,\dots,\gamma_n\in H^*_{S^1}(\bP^1;\bC)$ and nonnegative integers $a_1,\dots, a_n$, we define genus-$0$, degree-$\beta$, $S^1$-equivariant descendant 
Gromov-Witten invariants of $\bP^1$:
\[
    \langle \tau_{a_1}(\gamma_1)\dots\tau_{a_n}(\gamma_n)\rangle_{0,n,\beta}^{\bP^1,S^1} := \int_{[\overline{\cM}_{0,n}(\bP^1,\beta)^{S^1}]^{\vir}}
    \frac{\iota^*(\prod_{i=1}^n \ev_i^*(\gamma_i)(\psi_i^{S^1})^{a_i})}{e_{S^1}(N^\vir)}\in \bC(\sv),
\]
where $\overline{\cM}_{0,n}(\bP^1,\beta)^{S^1}$ is the $S^1$-fixed locus, $e_{S^1}(N^\vir)$ is the $S^1$-equivariant Euler class of the virtual normal bundle of $\overline{\cM}_{0,n}(\bP^1,\beta)$,
and $\iota:\overline{\cM}_{0,n}(\bP^1,\beta)^{S^1} \hookrightarrow \overline{\cM}_{0,n}(\bP^1,\beta)$ is the inclusion map.
The genus-$0$, degree-$\beta$, $S^1$-equivariant primary 
Gromov-Witten invariants of $\bP^1$ is defined as 
\[
\langle \gamma_1\dots\gamma_n\rangle_{0,n,\beta}^{\bP^1,S^1} := \langle \tau_{0}(\gamma_1)\dots\tau_{0}(\gamma_n)\rangle_{0,n,\beta}^{\bP^1,S^1}.
\]

Let $\bt = t^0 1+ t^1H$, we define the following double correlator:
\[
    \llangle \tau_{a_1}(\gamma_1),\dots,\tau_{a_n}(\gamma_n)\rrangle_{0,n}^{\bP^1,S^1} := 
    \sum_{\beta\in E(\bP^1)}\sum_{m=0}^\infty \frac{1}{m!}\langle \tau_{a_1}(\gamma_1),\dots,\tau_{a_n}(\gamma_n)
    ,\bt^m\rangle_{0,n+m,\beta}^{\bP^1,S^1}.
\]

For $j=1,\dots,n$, introduce formal variables
\[
    {\bf{u}}_j = {\bf{u}}_j(z) = \sum_{a\geq 0}(u_j)_az^a
\]
where $(u_j)_a\in H^*_{S^1}(\bP^1)\otimes_{\bC[\sv]}\bC(\sv)$. Define
\[
\llangle {\bf u}_1,\dots, {\bf u}_n \rrangle_{0,n}^{\bP^1,S^1} =  \llangle {\bf u}_1(\psi),\dots, {\bf u}_n(\psi)\rrangle_{0,n}^{\bP^1,{S^1}}
        = \sum_{a_1,\dots,a_n\geq 0}\llangle (u_1)_{a_1}\psi^{a_1},\dots, (u_n)_{a_n}\psi^{a_n}\rrangle_{0,n}^{\bP^1,S^1}.
\]

Let $z_1,\dots,z_n$ be formal variables and $\gamma_1,\dots,\gamma_n\in H^*_{S^1}(\bP^1)\otimes_{\bC[\sv]}\bC(\sv)$. Define
\[
    \llangle \frac{\gamma_1}{z_1-\psi},\dots,\frac{\gamma_n}{z_n-\psi}\rrangle_{0,n}^{\bP^1,S^1} = 
    \sum_{a_1,\dots,a_n\in\bZ_{\geq 0}}\llangle\gamma_1\psi^{a_1},\dots,\gamma_n\psi^{a_n}\rrangle_{0,n}^{\bP^1,S^1}\prod_{i=1}^{n}z_i^{-a_i-1}.
\]
We use the conventions that
\[
    \langle\frac{\gamma}{z-\psi}\rangle_{0,1,0}^{\bP^1,S^1} := z\int_{\bP^1}\gamma,
\]
\[
    \langle\frac{\gamma_1}{z-\psi},\gamma_2\rangle_{0,2,0}^{\bP^1,S^1} := \int_{\bP^1}\gamma_1\cup\gamma_2,
\]
\[
    \langle\frac{\gamma_1}{z_1-\psi_1},\frac{\gamma_2}{z_2-\psi_2}\rangle_{0,2,0}^{\bP^1,S^1} := \frac{1}{z_1+z_2}\int_{\bP^1}\gamma_1\cup\gamma_2.
\]

\subsection{$S^1$-fixed locus and decorated graphs}
The components of the $S^1$-fixed locus of the moduli space $\overline{\cM}_{0,n}(\CP,\beta)$ can be described by the decorated graphs 
introduced in \cite[Definition 52]{Liu13}, defined as follows.

\begin{defn}[Decorated graphs]\rm
    Define $G_{0,n}(\CP,\beta)$ to be the set of all decorated graphs $\vec{\Ga}=(\Ga,\vec{f},\vec{d},\vec{s})$ defined as follows. 
    Let $n\in\bZ_{\geq 0}$ and $\beta=d[\bP^1]\in E(\CP)$. A genus-$0$, $n$-pointed,
    degree $\beta$ decorated graph for $\CP$ is a tuple $\vec{\Ga}=(\Ga,\vec{f},\vec{d},\vec{s})$ consisting of 
    the following data.
    \begin{enumerate}
        \item $\Gamma$ is a compact, connected 1-dimensional CW complex. Let $V(\Ga)$ denote
        the set of vertices in $\Ga$. 
        Let $E(\Ga)$ denote the set of edges, where an edge $e$ is a line connecting two vertices.
        Let $F(\Ga)$ be the set of flags:
        \[
            \{(e,v)\in E(\Ga)\times V(\Ga):v\in e\}.
        \]
        For each $v\in V(\Ga)$, let $E_v$ denote the edges attached to $v$, and let $\val(v)=|E_v|$ denote the number of edges incident to $v$.
        \item The \emph{label map} $\vec{f}: V(\Ga)\rightarrow \{1,2\}$ labels each vertex with a number.
            If $v_1,v_2\in V(\Ga)$ are connected by an edge, we require
            $\vec{f}(v_1)\neq \vec{f}(v_2)$.
        \item The \emph{degree map} $\vec{d}:E(\Ga)\rightarrow \bZ_{>0}$ sends an edge $e$ to a positive integer $\vec{d}(e)=d_e$.
        \item The \emph{marking map} $\vec{s}:\{1,2,\dots,n\}\rightarrow V(\Ga)$. For each $v\in V(\Ga)$, define $S_v := \vec{s}^{-1}(v)$, and $n_v=|S_v|$.
    \end{enumerate}
    The data is required to satisfy the following conditions:
    \begin{itemize}
        \item [(i)] The graph $\Ga = (V(\Ga),E(\Ga))$ is a tree:
        \[
            |E(\Ga)| - |V(\Ga)| + 1 = 0.
        \]
        \item [(ii)] (degree) $d = \sum_{e\in E(\Ga)}d_e$.
    \end{itemize}
\end{defn}

Given $\vec{\Ga}\in G_{0,n}(\CP,\beta)$, we introduce the following notations:
\begin{itemize}
    \item (weight) We define
    \[
        {\bf w}(p_1) = -\sv, \quad {\bf w}(p_2) = \sv,
    \]
    For a flag $f=(e,v)\in F_v$, we define
    \begin{equation*}
        {\bf w}_f := \frac{{\bf w}(p_{\vec{f}(v)})}{d_e}.
    \end{equation*}
    \item (edge contribution) For each edge $e\in E(\Ga)$ and $d\in \bZ_{>0}$, we define
        \[
            {\bf h}(e,d) = \frac{(-1)^dd^{2d}}{(d!)^2\sv^{2d}}.
        \]
\end{itemize}

By \cite[Theorem 73]{Liu13}, we get
\begin{prop}\label{prop:localization-liu} Let $\beta=d[\CP]\in E(\CP)$. Then for $\gamma_1,\dots,\gamma_{n}\in H^*_{S^1}(\CP)$ and $a_1,\dots, a_n\in\bZ_{\geq 0}$, we have
    \begin{equation*}
    \begin{aligned}
        &\langle \tau_{a_1}(\gamma_1)\dots\tau_{a_{n}}(\gamma_{n})\rangle^{\CP,S^1}_{0,n,\beta} 
            \\
            = &\sum_{\vec{\Ga}\in G_{0,n}(\CP,\beta)}\frac{1}{|\Aut(\vec{\Ga})|}\prod_{e\in E(\Ga)}\frac{{\bf h}(e,d_e)}{d_e}
            \prod_{v\in V(\Ga)}\Big({\bf w}(p_{\vec{f}(v)})^{|E_v|-1}\prod_{i\in S_v}i^*_{p_{\vec{f}(v)}}\gamma_{i}\Big)
            \\
            &\cdot\prod_{v\in V(\Ga)}\int_{\overline{\cM}_{0,E_v\cup S_v}}\frac{\prod_{i\in S_v}\psi_i^{a_i}}{\prod_{e\in E_v}({\bf w}_{(e,v)}-\psi_{(e,v)})}.
    \end{aligned}
\end{equation*}
    We use the following convention for the unstable integrals:
\begin{equation*}
    \begin{aligned}
        \int_{\overline{\cM}_{0,1}}\frac{1}{{\bf w}-\psi} = {\bf w},
        \quad
        \int_{\overline{\cM}_{0,2}}\frac{\psi_2^a}{{\bf w}-\psi_1} = (-{\bf w})^a&, \quad a\in \bZ_{\geq 0},
        \\
        \int_{\overline{\cM}_{0,2}}\frac{1}{({\bf w}_1-\psi_1)({\bf w}_2-\psi_2)}= \frac{1}{{\bf w}_1+{\bf w}_2}&.
    \end{aligned}
\end{equation*}
\end{prop}

\subsection{Disk invariants}\label{sec:disk}
Given a nonnegative integer $n$ and an element $\beta'=(d_-,d_+)\in E(\CP,L)$, $d_-\neq d_+$. 
Let $D$ be the disk and $\partial D$ be its boundary. Let $(D,\partial D, x_1,\dots,x_n)$ be 
the disk with $n$ interior marked points. A degree-$\beta'$ disk map with $n$ interior points is a holomorphic map
$u: (D,\partial D, x_1,\dots,x_n)\rightarrow (\CP,L)$
satisfying $u_*([D]) = \beta'$ and $u(\partial D)\subset L$. 

Let $\overline{\cM}_{(0,1),n}(\CP, L,\beta')$ be the moduli space of degree-$\beta'$ with $n$
interior points. Let $\ev_i: \overline{\cM}_{(0,1),n}(\CP,L,\beta')\rightarrow \CP$ be the evaluation map at the $i$-th marked point.
The $S^1$-action on $(\CP,L)$ induces the $S^1$-action on $\overline{\cM}_{(0,1),n}(\CP,L,\beta')$. 
Let $\cF:=\overline{\cM}_{(0,1),n}(\CP,L,\beta')^{S^1}$ be the $S^1$-fixed locus and $\iota: \cF\rightarrow \overline{\cM}_{(0,1),n}(\CP,L,\beta')$ be the inclusion. 
The evaluation map $\ev_i$ is $S^1$-equivariant. 

For $i=1,\dots,n$, let $\bL_i$ be the $i$-th tautological line bundle over $\overline{\cM}_{(0,1),n}(\CP,L,\beta')$ formed by the cotangent line at the $i$-th
marked point. Define the $i$-th descendant class $\psi_i$ as 
\[
    \psi_i := c_1(\bL_i)\in H^2(\overline{\cM}_{(0,1),n}(\CP,L,\beta');\bQ).
\]
We choose an $S^1$-equivariant lift $\psi_i^{S^1}\in H^2_{S^1}(\overline{\cM}_{(0,1),n}(\CP, L,\beta');\bQ)$ of $\psi_i$.

Let $\gamma_1,\dots,\gamma_n\in H^*_{S^1}(\CP,\bC)$ and $a_1,\dots, a_n\in\bZ_{\geq 0}$. We define the degree-$\beta'$, $S^1$-equivariant 
open Gromov-Witten disk invariants of $(\CP,L)$
\begin{equation*}
    \langle\tau_{a_1}(\gamma_1)\dots\tau_{a_n}(\gamma_n)\rangle^{(\CP,L),S^1}_{(0,1),\beta'} 
    := \int_{[\cF]^\vir}\frac{\iota^*(\prod_{i=1}^{n}\ev_i^*(\gamma_i)(\psi_i^{S^1})^{a_i})}{e_{S^1}(N^\vir)} \in \bC(\sv),
\end{equation*}
where $[\cF]^\vir$ is the virtual fundamental class of $\cF$, and $e_{S^1}(N^\vir)$ is the $S^1$-equivariant
Euler class of the virtual normal bundle of $\cF$ in $\overline{\cM}_{(0,1),n}(\CP,L,\beta')$. Since $\cF$ is a compact orbifold without boundary, the above integral is well-defined.

\subsection{Localization formula of disk invariants}
The disk invariants can be computed by localization formula. We introduce the following notations.
\begin{itemize}
    \item (disk factor) For $\mu\in\bZ_{>0}$, we define the disk factors as 
    \[
        D^1(\mu) = (-1)^{\mu+1}\frac{\mu^{\mu-2}}{\mu!\sv^{\mu-2}}, \quad D^2(\mu) = \frac{\mu^{\mu-2}}{\mu!\sv^{\mu-2}}.
    \]
    For $\mu\in \bZ_{\neq 0}$, we define 
    \begin{equation*}
        D(\mu) = \left\{
        \begin{aligned}
            &D^1(-\mu), &\mu < 0;
            \\
            &D^2(\mu), &\mu>0.
        \end{aligned}
        \right.
    \end{equation*}
    \item (insertion) For $\mu\in \bZ_{\neq 0}$, we define 
    \begin{equation*}
        h(\mu) = \left\{
        \begin{aligned}
            &1, & \mu < 0;
            \\
            &2, & \mu> 0.
        \end{aligned}
        \right.
    \end{equation*}
    \item We consider the following decomposition: 
    \[
        G_{0,n+1}(\CP,\beta) = G^1_{0,n+1}(\CP,\beta)\sqcup G^2_{0,n+1}(\CP,\beta),
    \]
    where
    $
        G_{0,n+1}^i(\CP,\beta) = \{\vec{\Ga}\in G_{0,n+1}(\CP,\beta): \vec{f}\circ \vec{s}(n+1) = i\}, \ i=1,2.
    $
    \item 
    The indicator function $\delta_{v,n+1}$ is defined as 
    \begin{equation*}
        \delta_{v,n+1} := \left\{
        \begin{aligned}
            &1, \quad \text{if } v=\vec{s}(n+1),
            \\
            &0, \quad \text{otherwise}.
        \end{aligned}
        \right.
    \end{equation*}
        
\end{itemize}

By the virtual localization formula in \cite{BNPT22}, we get the following proposition.
\begin{prop}\label{prop:open-localization}
    Let $\beta'=(d_-,d_+)\in E(\CP,L)$ with $d_-\neq d_+$. Let $d=\min\{d_-,d_+\}$, $\beta=d[\CP]\in E(\CP)$ and $\mu = d_+-d_-$. Then for $\gamma_1,\dots,\gamma_n\in H_{S^1}^*(\CP)$ and $a_1,\dots,a_n\geq 0$, we have
    \begin{equation*}
        \begin{aligned}
            &\langle  \tau_{a_1}(\gamma_1)\dots\tau_{a_n}(\gamma_n)\rangle_{(0,1),\beta'}^{(\CP,L),S^1}
            \\
            = &\sum_{\vec{\Ga}\in G^{h(\mu)}_{0,n+1}(\CP,\beta)} \frac{1}{|\Aut(\vec{\Ga})|}\prod_{e\in E(\Ga)}\frac{{\bf h}(e,d_e)}{d_e}
            \prod_{v\in V(\Ga)}\Big({\bf w}(p_{\vec{f}(v)})^{|E_v|-1}\prod_{i\in S_v\backslash \{n+1\}}i^*_{p_{\vec{f}(v)}}\gamma_{i} \Big)
            \\
            &\cdot D(\mu)\Big(\frac{\mu}{\sv}\Big)\prod_{v\in V(\Ga)}\int_{\overline{\cM}_{0,E_v\cup S_v}}\frac{\prod_{i\in S_v\backslash \{n+1\}}\psi_i^{a_i}}{(\frac{\sv}{\mu}-\psi_{n+1})^{\delta_{v,n+1}}\prod_{e\in E_v}({\bf w}_{(e,v)}-\psi_{(e,v)})}.
        \end{aligned}
    \end{equation*}
\end{prop}

By Proposition \ref{prop:localization-liu} and Proposition \ref{prop:open-localization}, we get the following theorem:

\begin{theorem}\label{thm:open-descendant} Let $\beta'=(d_-,d_+)\in E(\CP,L)$ with $d_-\neq d_+$. Let $d=\min\{d_-,d_+\}$, $\beta=d[\CP]\in E(\CP)$ and $\mu = d_+-d_-$. Then for $\gamma_1,\dots,\gamma_n\in H_{S^1}^*(\CP)$ and $a_1,\dots,a_n\geq 0$, we have
    \begin{equation*}
        \begin{aligned}
            & \langle \tau_{a_1}(\gamma_1)\dots\tau_{a_n}(\gamma_n)\rangle^{(\CP,L),S^1}_{(0,1),\beta'}
            \\
            & \ \ \ \ \ = D(\mu)\cdot \int_{[\overline{\cM}_{0,n+1}(\bP^1,\beta)]^{\vir}}\frac{\ev_{n+1}^* \phi_{h(\mu)}\prod_{i=1}^{n}\psi_i^{a_i}\ev_i^*(\gamma_i)}
        {\frac{\sv}{\mu}(\frac{\sv}{\mu}-\psi_{n+1})}.
        \end{aligned}
\end{equation*}
\end{theorem}

\subsection{Equivariant $J$-function of $\CP$}\label{sec:J-function}
The $S^1$-equivariant $J$-function $J_{\CP}(z)$ is characterized by
\[
    J_{\CP}(z) = 1 + \sum_{\alpha\in\{1,2\}}\llangle 1,\frac{{\phi}_\alpha}{z-\psi}\rrangle_{0,2}^{\bP^1,S^1}{\phi}^\alpha,
\]
where $\{\phi^\alpha\}$ is the dual basis of $\{\phi_\alpha\}$ with respect to $S^1$-equivariant Poincaré pairing $(\cdot,\cdot)_{\CP,S^1}$.
By the genus zero mirror theorem \cite{G96,LLY97}, 
\[
    J_{\CP}(z) = e^{(t^0+t^1H)/z}\Bigg(
        1 + \sum_{d=1}^{\infty}\frac{q^d}{\prod_{m=1}^{d}(H+\sv/2+mz)\prod_{m=1}^{d}(H-\sv/2+mz)}
    \Bigg),
\]
where $q=e^{t^1}$.

Let $J_{\CP}(z) = J_{\CP}^1\phi_1+J_{\CP}^2\phi_2$. Then for $\alpha=1,2$, we have 
\begin{equation}\label{eq:J-function}
    \begin{aligned}
        J_{\CP}^\alpha &= e^{(t^0+ t^1\Delta^\alpha/2)/z}\sum_{d=0}^{\infty}
        \frac{q^d}{d!z^d}\frac{1}{\prod_{m=1}^{d}(\Delta^\alpha+mz)}
        \\
        &= e^{(t^0+ t^1\Delta^\alpha/2)/z}\sum_{m=0}^{\infty}\Big(\frac{\sqrt{q}}{z}\Big)^{2m}
        \frac{\Ga(\Delta^\alpha/z+1)}{m!\Ga(\Delta^\alpha/z+m+1)}
        \\
        &= e^{t^0/z}z^{\Delta^\alpha/z}\Ga(\Delta^\alpha/z+1)I_{\Delta^\alpha/z}\Big(\frac{2\sqrt{q}}{z}\Big),
    \end{aligned}
\end{equation}
where 
$$\Delta^1=-\sv,\quad \Delta^2=\sv,$$ and
the function $I_\alpha(x)$ is the \emph{modified Bessel function of first kind} in Appendix \ref{sec:Bessel}.

\subsection{The disk potential}\label{sec:disk-potential}
We introduce the following conventions for $\beta'\in E(\bP^1, L)$:
\begin{quote}
    Let $\beta'=(d_-,d_+)\in E(\CP,L)$, $d:=\min\{d_-,d_+\}$, $\beta:=d[\CP]\in E(\CP)$ and $\mu := d_+-d_-$.
\end{quote}
Let $\bt = t^01+t^1H$ and consider the following generating function of disk invariants of $(\CP,L)$:
\[
    F_{0,1}^{\OGW}(\bt;X) = \sum_{\beta'\in E(\CP,L)\atop {\mu \in \bZ_{\neq 0}}} \sum_{l\geq 0}
    \frac{1}{l!}\langle \bt^l\rangle_{(0,1),\beta'}^{\OGW} X^\mu.
\] 
By Theorem \ref{thm:open-descendant},
\begin{equation*}
    \begin{aligned}
        &  F_{0,1}^{\OGW}(\bt;X) =
        \\
        &\ \ \ \ \ = \sum_{\beta\in E(\bP^1)}\sum_{l\geq 0}\frac{1}{l!}\sum_{\mu\in\bZ_{\neq 0}}
        \langle\bt^l,\frac{\phi_{h(\mu)}}{\frac{\sv}{\mu}(\frac{\sv}{\mu}-\psi)}\rangle^{\bP^1,S^1}_{0,l+1,\beta}D(\mu)X^\mu
        \\
        &\ \ \ \ \ = \sum_{\mu\in\bZ_{\neq 0}}\Big(\frac{1}{\Delta^{h(\mu)}}+\llangle 1,\frac{\phi_{h(\mu)}}{\frac{\sv}{\mu}-\psi}\rrangle^{\CP,S^1}_{0,2}\Big)D(\mu)X^\mu
        \\
        &\ \ \ \ \ = \sum_{\mu>0}\Big(\left(J_{\CP}\right)_{1}(-\sv/\mu)D^1(\mu)X^{-\mu} + {(J_{\CP})}_2(\sv/\mu)D^2(\mu)X^\mu\Big),
    \end{aligned}
\end{equation*}
where $(J_{\CP})_\alpha(z):= ({J_{\CP}}(z),\phi_\alpha)_{\CP,S^1},\alpha=1,2$ are the components of the $J$-function in Section \ref{sec:J-function}.

By Equation (\ref{eq:J-function}), for $\mu>0$
\begin{equation*}
    \begin{aligned}
        J_{\CP}^1(-\sv/\mu)&= -\sv (J_{\CP})_1(-\sv/\mu)
        = e^{-\mu t^0/\sv}(-\sv/\mu)^\mu\Ga(\mu+1)I_\mu(-2\sqrt{q}\mu/\sv),
    \\
        J_{\CP}^2(\sv/\mu) &= \sv (J_{\CP})_2(\sv/\mu)
        = e^{\mu t^0/\sv}(\sv/\mu)^\mu\Ga(\mu+1)I_\mu(2\sqrt{q}\mu/\sv).
    \end{aligned}
\end{equation*}
We get
\[
    F_{0,1}^{\OGW}(\bt;X) = \sum_{\mu>0}e^{-\mu t^0/\sv}\frac{\sv}{\mu^2}I_\mu(-2\sqrt{q}\mu/\sv)X^{-\mu}
    + \sum_{\mu>0}e^{\mu t^0/\sv}\frac{\sv}{\mu^2}I_\mu(2\sqrt{q}\mu/\sv)X^\mu.
\]
Let $q, \sv$ be positive real numbers. By the symmetry of the modified Bessel function $I_\alpha(x)$ (see Appendix \ref{sec:Bessel}), we have
\begin{equation}\label{eqn:explicit-disk}
        F_{0,1}^{\OGW}(\bt;X) = \sum_{\mu\in\bZ_{\neq 0}}e^{\mu t^0/\sv}\frac{\sv}{\mu^2}I_\mu(2\sqrt{q}\mu/\sv)X^\mu.
\end{equation}

\section{Gromov-Witten theory of $\cS$}\label{sec:closed-gw-X}

\subsection{Equivariant Gromov-Witten invariants of $\cS$}
Given a nonnegative integer $n$ and an effective curve class $\beta\in E(\cS)$, let $\Mbar_{0,n}(\cS,\beta)$ be the moduli space of genus-0, $n$-pointed, degree-$\beta$ stable maps 
to $\cS$. Let $\ev_i:\Mbar_{0,n}(\cS,\beta)\rightarrow \cS$ be the evaluation map at the $i$-th marked point. The $T$-action on $\cS$ induces a $T$-action on the moduli space
$\Mbar_{0,n}(\cS,\beta)$ and the evaluation map $\ev_i$ is $T$-equivariant. Let $\Mbar_{0,n}(\cS,\beta)^T$ be the $T$-fixed locus of $\Mbar_{0,n}(\cS,\beta)$, and $
\iota:\Mbar_{0,n}(\cS,\beta)^T\rightarrow \Mbar_{0,n}(\cS,\beta)$ be the inclusion.

For $i=1,\dots,n$, let $\bL_i$ be the $i$-th tautological line bundle over $\Mbar_{0,n}(\cS,\beta)$ formed by the cotangent line at the $i$-th marked point. Define the $i$-th
descendant class $\psi_i$ as 
\[
    \psi_i := c_1(\bL_i)\in H^2(\Mbar_{0,n}(\cS,\beta);\bQ).
\]
We choose a $T$-equivariant lift $\psi_i^T\in H^2_T(\Mbar_{0,n}(\cS,\beta);\bQ)$ of $\psi_i$.

Let $\gamma_1,\dots,\gamma_n\in H^*_T(\cS;\bC)$ and $a_1,\dots,a_n\in\bZ_{\geq 0}$. We define the genus-0, $n$-pointed, degree-$\beta$, $T$-equivariant descendant 
Gromov-Witten invariant
\begin{equation*}
    \begin{aligned}
        & \langle \tau_{a_1}(\gamma_1)\dots\tau_{a_n}(\gamma_n)\rangle_{0,\beta}^{\cS,T}
        \\
        & := \int_{[\Mbar_{0,n}(\cS,\beta)^T]^{\vir,T}}\frac{\iota^*(\prod_{i=1}^{n}\ev_i^*(\gamma_i)(\psi_i^T)^{a_i})}{e_T(N^\vir)}\in\bC(\su_1,\su_2),
    \end{aligned}
\end{equation*}
where $[\Mbar_{0,n}(\cS,\beta)^T]^{\vir,T}$ is the virtual fundamental class, and $e_T(N^\vir)$ is the $T$-equivariant Euler class of the virtual normal bundle of $\Mbar_{0,n}(\cS,\beta)^T$ in $\Mbar_{0,n}(\cS,\beta)$.

\subsection{Equivariant $J$-function of $\cS$}\label{sec:JX}
Let $\btau =\btau_0 +\btau_2\in H^*_T(\cS)\otimes_{\bC[\su_1,\su_2]}\bC(\su_1,\su_2)$, where $\btau_0 = \tau_01\in H^0_T(\cS)$ and $\btau_2 = \tau_1H^T_1 + \tau_2H^T_2\in H^2_T(\cS)$. We define
\[
    \llangle \tau_{a_1}(\gamma_1),\dots,\tau_{a_n}(\gamma_n)\rrangle_{0,n}^{\cS,T} :=
    \sum_{\beta\in E(\cS)}\sum_{m=0}^\infty \frac{1}{m!}\langle\tau _{a_1}(\gamma_1),\dots,\tau_{a_n}(\gamma_n),\btau^m\rangle_{0,n+m,\beta}^{\cS,T}.
\]
Let $z_1,\dots,z_n$ be formal variables. We define
\[
    \llangle \frac{\gamma_1}{z_1-\psi},\dots,\frac{\gamma_n}{z_n-\psi}\rrangle_{0,n}^{\cS,T} = 
    \sum_{a_1,\dots,a_n\in\bZ_{\geq 0}}\llangle\gamma_1\psi^{a_1},\dots,\gamma_n\psi^{a_n}\rrangle_{0,n}^{\cS,T}\prod_{i=1}^{n}z_i^{-a_i-1}.
\]
Let $\{u_i\}_{i=1,2,3}$ be a basis of $H^*_T(\cS)\otimes_{\bC[\su_1,\su_2]}\bC(\su_1,\su_2)$. 
The $T$-equivariant $J$-function for $\cS$ is 
\[
    J_\cS(\btau,z) := 1 + \sum_{i=1}^3 \llangle 1,\frac{u_i}{z-\psi}\rrangle_{0,2}^{\cS,T}u^i,
\]
where $\{u^i\}$ is the dual basis of $\{u_i\}$ under the $T$-equivariant Poincaré pairing $(\cdot,\cdot)_{\cS,T}$.

\subsection{Equivariant $I$-function of $\cS$}\label{sec:IX}
\subsubsection{Genus zero mirror theorem}
Following \cite{G98, LLY99, LLY3}, the $T$-equivariant $I$-function of $\cS$ is defined as follows. Let 
\begin{equation*}
    \begin{aligned}
        &I_\cS(\bq,z) = e^{(\log q_0 + H_1^T\log q_1+H_2^T\log q_2)/z}\sum_{d_1,d_2\geq 0}q_1^{d_1}q_2^{d_2}
        \\
        & \ \ \ \ \ \cdot \frac{\prod_{m=-d_1+d_2}^{\infty}(D_1^T+(-d_1+d_2-m)z)}{\prod_{m=0}^{\infty}(D_1^T+(-d_1+d_2-m)z)}
        \cdot \frac{\prod_{m=d_1-d_2}^{\infty}(D_2^T+(d_1-d_2-m)z)}{\prod_{m=0}^{\infty}(D_2^T+(d_1-d_2-m)z)}
        \\[1.1ex]
        & \ \ \ \ \ \cdot \frac{\prod_{m=d_1}^{\infty}(D_3^T+(d_1-m)z)}{\prod_{m=0}^{\infty}(D_3^T+(d_1-m)z)}
        \cdot \frac{\prod_{m=d_2}^{\infty}(D_4^T+(d_2-m)z)}{\prod_{m=0}^{\infty}(D_4^T+(d_2-m)z)}.
    \end{aligned}
\end{equation*}
where $\bq=(q_0,q_1,q_2)$.

By \cite{G98, LLY99, LLY3}, we have the following genus zero mirror theorem.
\begin{theorem}\label{thm:genus-zero-mirror-theorem}
    Let $\tau_0(\bq) = \log q_0$, $\tau_1(\bq)= \log q_1$, $\tau_2(\bq)= \log q_2$. 
    Then we have
    \[
        e^{\frac{\tau_0(\bq)}{z}}J_\cS(\btau_2(\bq),z) = I_\cS(\bq,z),
    \]
    where the $I$-function is expanded in powers of $z^{-1}$:
    \[
        I_\cS(\bq,z) = 1 + z^{-1}(\log q_0 + \log q_1 H^T_1 + \log q_2 H^T_2) + o(z^{-1}).
    \]
\end{theorem}

\subsubsection{Analysis of $I$-function}
Let $(d_1,d_2)\in E(\cS)$, $d=\min\{d_1,d_2\}$ and $\mu = |d_1-d_2|\in\bZ_{\geq 0}$. We decompose the set $E(\cS)\cong \bZ^2_{\geq 0}$ into three subsets:
\begin{itemize}\setlength{\itemsep}{1.1ex}
    \item $E^1(\cS) = \{(d_1,d_2)\in \bZ^2_{\geq 0}: d_1 = d+\mu, \ d_2 = d \text{ for some } d\geq 0, \ \mu >0\};$
    \item $E^2(\cS) = \{(d_1,d_2)\in \bZ^2_{\geq 0}: d_1 = d, \ d_2 = d+\mu \text{ for some } d\geq 0, \ \mu >0\};$
    \item $E^3(\cS) = \{(d_1,d_2)\in \bZ^2_{\geq 0}: d_1 = d_2 = d \text{ for some } d\geq 0\}.$
\end{itemize}

Let $\iota_{\si_0}: \pZero \rightarrow \cS$ be the inclusion of $\pZero$ into the toric surface $\cS$. Consider the function
$$\iota^*_{\si_0}I_\cS(\bq,z) := I_\cS(\bq,z)|_\pZero.$$
According to the decomposition of the set $E(\cS)$, we have $\iota^*_{\si_0}I_\cS(\bq,z) = I^1 + I^2 + I^3$, where

\begin{align*}
        &I^1 = e^{(\log q_0)/z}\sum_{d\geq 0}
            \sum_{\mu > 0} \frac{q^{d+\mu}_1q_2^d}{d!(d+\mu)!z^{2d+\mu}}
            \\
            & \ \ \ \ \ \ \cdot 
            \frac{\prod_{m=-\mu}^{\infty}(-\su_2+(-\mu-m)z)}{\prod_{m=0}^{\infty}(-\su_2+(-\mu-m)z)}
            \frac{\prod_{m=\mu}^{\infty}(-\su_1+(\mu-m)z)}{\prod_{m=0}^{\infty}(-\su_1+(\mu-m)z)},
        \\
        &I^2 = e^{(\log q_0)/z}\sum_{d\geq 0}
            \sum_{\mu > 0} \frac{q_1^dq_2^{d+\mu}}{d!(d+\mu)!z^{2d+\mu}}
            \\
            & \ \ \ \ \ \ \cdot
            \frac{\prod_{m=\mu}^{\infty}(-\su_2+(\mu-m)z)}{\prod_{m=0}^{\infty}(-\su_2 + (\mu-m)z)}
            \frac{\prod_{m=-\mu}^{\infty}(-\su_1+(-\mu-m)z)}{\prod_{m=0}^{\infty}(-\su_1+(-\mu-m)z)},
        \\
        &I^3 = e^{(\log q_0)/z}\sum_{d\geq 0} \frac{q_1^dq_2^d}{(d!)^2z^{2d}}.
\end{align*}
Let $I^i(\bq;\sv,z) := I^i\big|_{\su_2=-\su_1=\sv}$, $i=1,2,3$. Then we have
% Restrict $I_i$, $i=1,2,3$ to $\su_1+\su_2=0$, $\su_1 = -\sv$, we have
\begin{align*}
           I^1(\bq;\sv,z) &= e^{(\log q_0)/z}\sum_{d\geq 0}
            \sum_{\mu > 0} \frac{q^{d+\mu}_1q_2^d}{d!(d+\mu)!z^{2d+\mu}}
            \frac{\prod_{m=-\mu}^{-1}(-\sv+(-\mu-m)z)}{\prod_{m=0}^{\mu-1}(\sv+(\mu-m)z)}
            \\
            &= e^{(\log q_0)/z}\sum_{d\geq 0}
            \sum_{\mu > 0} \frac{q^{d+\mu}_1q_2^d}{d!(d+\mu)!z^{2d+\mu}}
            \frac{(-1)^\mu\sv}{\sv+\mu z},
            \\
            I^2(\bq;\sv,z) &= e^{(\log q_0)/z}\sum_{d\geq 0}
            \sum_{\mu > 0} \frac{q_1^dq_2^{d+\mu}}{d!(d+\mu)!z^{2d+\mu}}
            \frac{\prod_{m=-\mu}^{-1}(\sv+(-\mu-m)z)}{\prod_{m=0}^{\mu-1}(-\sv + (\mu-m)z)}
            \\
            &= e^{(\log q_0)/z}\sum_{d\geq 0}
            \sum_{\mu > 0} \frac{q_1^dq_2^{d+\mu}}{d!(d+\mu)!z^{2d+\mu}}
            \frac{(-1)^\mu\sv}{\sv-\mu z},
            \\
            I^3(\bq;\sv,z) &= e^{(\log q_0)/z}\sum_{d\geq 0} \frac{q_1^dq_2^d}{(d!)^2z^{2d}}.
\end{align*}

In the following paragraphs, we view
$\sv$ as a formal variable and expand $I^i(\bq;\sv,z)$ in powers of $\sv^{-1}$ by the following equations:
\begin{equation}\label{eqn:expansion}
    \frac{\sv}{\sv+\mu z} = \sum_{k=0}^{\infty} (-1)^k\big(\frac{\mu}{\sv}\big)^kz^k,
    \quad
    \frac{\sv}{\sv-\mu z} = \sum_{k=0}^{\infty} \big(\frac{\mu}{\sv}\big)^kz^k.
\end{equation}

Let $[z^{-2}]I^i,\ i=1,2,3$ be the $z^{-2}$-coefficients of the above expansion of $I^i(q;\sv,z)$. We have
\begin{align}\label{eqn:I1}
        I^1(\bq;\sv,z) &= \sum_{l=0}^{\infty}
        \frac{(\log q_0)^l}{l!z^l}\sum_{d\geq 0,{\mu >0}}\frac{q_1^{d+\mu}q_2^d}{d!(d+\mu)!z^{2d+\mu}}(-1)^\mu\sum_{k=0}^{\infty}(-1)^k\big(\frac{\mu}{\sv}\big)^k z^k. \notag
        \\
        [z^{-2}]I^1(\bq;\sv,z) &= - q_1\sv + 
        \sum_{d\geq 0,{\mu > 0}}\sum_{l=0}^{\infty}\frac{(-\log q_0)^l}{l!}
        \frac{q_1^{d+\mu}q_2^d}{d!(d+\mu)!}\big(\frac{\mu}{\sv}\big)^{2d+l+\mu-2} 
        \\
        &= - q_1\sv + \sum_{d\geq 0,{\mu > 0}}e^{-(\mu\log q_0)/\sv}\frac{q_1^{d+\mu}q_2^d}{d!(d+\mu)!}\big(\frac{\mu}{\sv}\big)^{2d+\mu-2}, \notag
\end{align}
where $q_1\sv$ is from the exceptional term $(l,d,\mu,k)=(0,0,1,-1)$. Similarly, we have
\begin{align}\label{eqn:I2I3}   
        \quad \quad \quad I^2(\bq;\sv,z) &= \sum_{l=0}^{\infty}
    \frac{(\log q_0)^l}{l!z^l}\sum_{d\geq 0,\mu > 0}
    \frac{q_1^dq_2^{d+\mu}}{d!(d+\mu)!z^{2d+\mu}}(-1)^\mu\sum_{k=0}^{\infty} \big(\frac{\mu}{\sv}\big)^kz^k, \notag
    \\
        [z^{-2}]I^2(\bq;\sv,z) &= q_2\sv + \sum_{d\geq 0, {\mu>0}}\sum_{l=0}^{\infty}\frac{(\log q_0)^l}{l!}
         \frac{q_1^dq_2^{d+\mu}(-1)^\mu}{d!(d+\mu)!}\big(\frac{\mu}{\sv}\big)^{2d+l+\mu-2} 
        \\
        &= q_2\sv + \sum_{d\geq 0, {\mu>0}}e^{(\mu \log q_0)/\sv} \frac{q_1^dq_2^{d+\mu}(-1)^\mu}{d!(d+\mu)!}\big(\frac{\mu}{\sv}\big)^{2d+\mu-2}, \notag
        \\
        [z^{-2}]I^3(\bq;\sv,z) &= \frac{\log^2 q_0}{2} + q_1q_2. \notag
\end{align}
\begin{remark}
	We would like to give a remark on the expansion in Equation \eqref{eqn:expansion}. In Theorem \ref{thm:genus-zero-mirror-theorem}, $I_\cS$ is expanded as a power series of $z^{-1}$ in order to match $J_\cS$. On the other hand, in the expansion in Equation \eqref{eqn:expansion}, positive powers of $z$ appear. It turns out that the expansion in Equation \eqref{eqn:expansion} is the correct one in the open/closed duality (Theorem \ref{thm:main-thm-i}). This expansion can either be explained as the asymptotic expansion of $I^i$ as $\sv\rightarrow\infty$ (Appendix \ref{sec:asymptotic}) or be explained algebraically as formal expansion (Section \ref{sec:formal expansion}). 
 
\end{remark}

\section{Open/closed correspondence}\label{sec:open-closed}
\subsection{The open/closed correspondence}
In this section, we prove the open/closed correspondence by relating the $I$-function $I_\cS$ to the disk potential $F_{0,1}^{\OGW}$. We refer the readers to Appendix \ref{sec:asymptotic} for the details of asymptotic expansion of the $I$-function.
\begin{theorem}\label{thm:main-thm-i}
    Under the relation $\log q_0 = t^0$, $q_1 = -\sqrt{q}X^{-1}$ and $q_2 = -\sqrt{q}X$,
    we have
    \begin{equation}\label{eqn:main}
        F_{0,1}^{\OGW} (\bt;X) = [z^{-2}]\big(I_\cS(\bq,z), \su_1\widetilde{\phi}_0\big)_{\cS,T}\Big|_{\su_2=-\su_1=\sv} + \text{Exc},
   \end{equation}
    where the $I$-function is in the asymptotic expansion as $\sv \rightarrow \infty$, and the exceptional term is 
    $\text{Exc} := -\sqrt{q}X^{-1} + \sqrt{q}X- \frac{(t^0)^2}{2\sv} - q\sv^{-1}$.
\end{theorem}

\begin{proof}
Consider the change of variables:
\[
    \log q_0 \mapsto t^0,\quad q_1 \mapsto -\sqrt{q}X^{-1}, \quad q_2\mapsto -\sqrt{q}X.
\]
Then by \eqref{eqn:I1} \eqref{eqn:I2I3}, we have
\begin{align*}
        &[z^{-2}]I^1(\bq(\bt,X);\sv,z)
        \\
        & \ \ \ \ \ =
        \sqrt{q}X^{-1}\sv + \sum_{d\geq 0, \mu >0}e^{-\mu t^0/\sv}\frac{\sqrt{q}^{2d+\mu}(-X)^{-\mu}}{d!(d+\mu)!}
        \big(\frac{\mu}{\sv}\big)^{2d+\mu-2}
        \\
        & \ \ \ \ \ =
        \sqrt{q}X^{-1}\sv + \sv \sum_{\mu>0}e^{-\mu t^0/\sv}\frac{\sv}{\mu^2}I_{\mu}(-2\sqrt{q}\mu/\sv)X^{-\mu},
        \\
        & [z^{-2}]I^2(\bq(\bt,X);\sv,z) 
        \\
        &\ \ \ \ \ = -\sqrt{q}X\sv + \sum_{d\geq 0,\mu>0} e^{\mu t^0/\sv}\frac{\sqrt{q}^{2d+\mu} X^\mu}{d!(d+\mu)!}\big(\frac{\mu}{\sv}\big)^{2d+\mu-2}
        \\
        &\ \ \ \ \ = -\sqrt{q}X\sv + \sv\sum_{\mu>0} e^{\mu t^0/\sv}\frac{\sv}{\mu^2} I_\mu(2\sqrt{q}\mu/\sv)X^\mu,
        \\
        &[z^{-2}]I^3(\bq(\bt,X);\sv,z) = \frac{(t^0)^2}{2} + q.
\end{align*}
By the explicit formula of $S^1$-equivariant disk potential $F_{0,1}^{\OGW}$ of $(\CP,L)$
in \eqref{eqn:explicit-disk}, we have
\[
    F_{0,1}^{\OGW}(\bt;X) = [z^{-2}]\Big(I_\cS(\bq(\bt,X),z),-\sv\widetilde{\phi}_0\Big)_{\cS,T}\Big|_{\su_2=-\su_1=\sv} + \text{Exc}.
\]
\end{proof}

\begin{remark}
    Theorem \ref{thm:main-thm-i} has a similar philosophy as \cite[Theorem 6.6]{LY22}.
    The formal variable $X$, which records the winding number, contributes $q_1, q_2$ in $I_\cS(\bq)$. 
    Via the mirror map, $q_1, q_2$ are related to the variables $\tau_1,\tau_2$, which encodes the closed invariants associated with $H_1^T, H_2^T$.
    
    Considering the pullback morphism $\iota^*: H^2_T(\cS)\rightarrow H^2_{S^1}(\bP^1)$  (see Equation \eqref{eqn:pullback}),
    we see that $\tau_1+\tau_2$ corresponds to the coordinate $t^1$ on $H^2(\bP^1)$, and we expect that $(\tau_2-\tau_1)/2$ corresponds to the open-sector variable on $H^1(L)$. 
    This suggests identifying the formal variable $X$
    with the closed variables via $$q_1 = -\sqrt{q}X^{-1}, \quad q_2=-\sqrt{q}X.$$
    
    This identification is consistent with the spirit of the change of variables in \cite[Theorem 6.6, Proposition 6.14]{LY22}.
\end{remark}

\subsection{Formal expansion of the $I$-function}\label{sec:formal expansion}
In this subsection, we give another explanation on the right hand side of \eqref{eqn:main} via algebraic method.
We introduce the following notations:
\begin{equation*}
    \begin{aligned}
        \cR_0 &:= \bC\left[\frac{\sv}{\sv+\mu z},\frac{\sv}{\sv-\mu z}\right][\![z^{-1},q_1,q_2,\log q_0]\!],
        \\
        \cR_1 &:= \bC[\![z^{-1},\sv,q_1,q_2,\log q_0]\!],
        \\
        \cR_2 &:= \bC(\!(z^{-1})\!)[\![q_1,q_2,\log q_0, \sv^{-1}]\!].
    \end{aligned}
\end{equation*}
\noindent
Formally, the function $I^i(\bq;\sv,z)$ lies in the ring $\cR_0$.
Let $\xi_1:\cR_0\rightarrow \cR_1$ be the map such that
\begin{equation*}
    \begin{aligned}
        \xi_1\Big(\frac{\sv}{\sv+\mu z}\Big) &= \frac{\sv}{\mu z}\Big( 1- \frac{\sv}{\mu z}+ (\frac{\sv}{\mu z})^2 +\dots\Big) ,
        \\
        \xi_1\Big(\frac{\sv}{\sv-\mu z}\Big) &= \frac{\sv}{-\mu z}\Big(1 + \frac{\sv}{\mu z} + (\frac{\sv}{\mu z})^2 + \dots \Big).
    \end{aligned}
\end{equation*}
Let $\xi_2:\cR_0\rightarrow \cR_2$ be the map such that
\begin{equation*}
    \begin{aligned}
        \xi_2\Big(\frac{\sv}{\sv+\mu z}\Big) &= 1- \frac{\mu z}{\sv}+ (\frac{\mu z}{\sv})^2 +\dots,
        \\
        \xi_2\Big(\frac{\sv}{\sv-\mu z}\Big) &= 1 + \frac{\mu z}{\sv} + (\frac{\mu z}{\sv})^2 + \dots.
    \end{aligned}
\end{equation*}

In Theorem \ref{thm:genus-zero-mirror-theorem} and Theorem \ref{thm:main-thm-i}, the functions $I^i(\bq;\sv,z)\in\cR_0$ are the global B-model encoding the information of A-model generating functions.
Theorem \ref{thm:genus-zero-mirror-theorem} states that 
\[
    \xi_1\Big(\iota^*_{\si_0}I_{\cS}(\bq,z)\Big|_{\su_2=-\su_1=\sv}\Big) = e^{\frac{\tau_0(\bq)}{z}}J_{\cS}(\btau_2(\bq),z)\Big|_{\pZero,\su_2=-\su_1=\sv}.
\]  
Our main result (Theorem \ref{thm:main-thm-i}) states that 
\[
    F_{0,1}^{(\CP,L),S^1}(\bt;X) = [z^{-2}]\xi_2\Big((I_\cS(\bq,z),\su_1\tilde{\phi}_0)_{\cS,T}\Big|_{\su_2=-\su_1=\sv}\Big) + Exc.
\]

\appendix
\section{Bessel functions}\label{sec:Bessel}
The special function $I_\alpha(x)$ in $J$-function is the modified Bessel function of the first kind. It is defined as
\[
    I_\alpha(x) = \sum_{m=0}^{\infty} \frac{1}{m!\Gamma(m+\alpha+1)}(\frac{x}{2})^{2m+\alpha}.
\]
For $n\in\mathbb{N}$, $I_n(x) = I_{-n}(x)$.

\section{Asymptotics of $I$-function}\label{sec:asymptotic}
Let's analyse the asymptotic behaviour of $I$-function in details. 
We consider the series
\begin{equation*}
    \begin{aligned}
            I^2(\bq;\sv,z) =&\ e^{(\log q_0)/z}\sum_{d\geq 0}
            \sum_{\mu > 0} \frac{q_1^dq_2^{d+\mu}}{d!(d+\mu)!z^{2d+\mu}}
            \frac{(-1)^\mu\sv}{\sv-\mu z},
\\
    \varphi_k(\bq,z) :=& \ e^{(\log q_0)/z}\sum_{d\geq 0}
            \sum_{\mu > 0} \frac{q_1^dq_2^{d+\mu}(-1)^\mu}{d!(d+\mu)!z^{2d+\mu}}
            \mu^k z^k , \quad (k \in \bZ_{\geq 0}).
    \end{aligned}
\end{equation*}
\noindent
We will show the following statements:
\begin{itemize}
    \item [(a)] $I^2(\bq;\sv,z)$ is pointwisely well-defined for all $\bq,\sv,z$, where $\{\sv\neq \mu z: \mu \in \mathbb{Z}_{\geq 1}\}$ and $z\neq 0$.
    \item [(b)] Analyse the limit behaviour of $I^2(\bq;\sv,z)$ as $\sv\rightarrow \infty$.
    \item [(c)] $\varphi_k(\bq,z)$ is  well-defined pointwisely for all $\bq, z$, where $z\neq 0$.
    \item [(d)] $\{\varphi_k(\bq,z)\sv^{-k}\}_{k=0}^\infty$ is an asymptotic series of $I^2(\bq;\sv,z)$ pointwisely as $\sv\rightarrow \infty$ in the following sense.
\begin{prop}\label{prop:asym-limit}
        For every $\bq,z>0$, there exists an increasing sequence $\{\sv_l\}_{l=1}^\infty$ satisfying $\sv_l\rightarrow \infty$ as $l\rightarrow \infty$, such that 
        $\lim_{l\rightarrow\infty} I^2(\bq;\sv_l,z) $ is convergent, and
        \[
            \lim_{l\rightarrow\infty} \frac{I^2(\bq;\sv_l,z)-\sum_{k=0}^{N-1}\varphi_k(\bq,z)\sv_l^{-k}}{\varphi_{N}(\bq,z)\sv_l^{-N}}  = 1.
        \] 
\end{prop}
    \item [(e)] View $\sv$ as a formal variable and show the $z^{-2}$-coefficient of the asymptotic series of $I^2(\bq;\sv,z)$ is well-defined. 
\end{itemize}

\

In step (a), fixing $\bq,\sv,z$, we have 
\begin{equation*}
    \lim_{\mu\rightarrow\infty} \Big|\frac{q_2^\mu}{\mu!z^\mu}\frac{\sv}{|\sv-\mu z|}\Big|^{1/\mu}=0,\quad
    \lim_{d\rightarrow \infty} \Big|\frac{q_1^dq_2^d}{d!z^{2d}}\Big|^{1/d}=0.
\end{equation*}
So the series $I^2(\bq;\sv,z)$ is absolutely convergent:
\begin{equation*}
    |I^2(\bq;\sv,z)| < e^{|\log q_0|/|z|}\sum_{d}\Big|\frac{q_1^dq_2^d}{d!z^{2d}}\Big|\sum_{\mu}\Big|\frac{q_2^\mu}{\mu !z^\mu}\frac{\sv}{|\sv-\mu z|}\Big| < \infty.
\end{equation*}

\

In step (b), we fix $\bq,z>0$.  
Notice that
$$I^2(\bq;\infty,z) := e^{(\log q_0)/z}\sum_{d,\mu} \frac{q_1^dq_2^{d+\mu}(-1)^\mu}{d!(d+\mu)!z^{2d+\mu}}$$ is absolutely convergent.
Let 
\begin{equation*}
    \begin{aligned}
        f_\sv(\bq,z;d,\mu):=& \ e^{(\log q_0)/z}\frac{q_1^dq_2^{d+\mu}}{d!(d+\mu)!z^{2d+\mu}}\frac{(-1)^\mu\sv}{\sv-\mu z},
        \\
        f_\infty(\bq,z;d,\mu) :=& \ e^{(\log q_0)/z}\frac{q_1^dq_2^{d+\mu}(-1)^\mu}{d!(d+\mu)!z^{2d+\mu}}.
    \end{aligned}
\end{equation*}
For fixed $\bq,z$, $f_\sv(\bq,z;d,\mu) \rightarrow f_\infty(\bq,z;d,\mu)$ for every $d,\mu$ pointwisely, as $\sv$ tends to infinity. 

We fix $z$ and then select a sequence $\{\sv_l\}_{l=1}^\infty \subset \bR_{>0}$ such that:
\begin{itemize}
    \item $\sv_l\rightarrow \infty \text{\ as \ }l\rightarrow +\infty$;
    \item There exists a linear function $s(\mu)$, such that $|\frac{\sv_l}{\sv_l - \mu z}| \leq s(\mu)$.
\end{itemize}
We can always find such $\sv_l$. For example, we assume $z>0$, if we choose $\sv_l = (l+1/2)z$, then 
\[
    \Big|\frac{\sv_l}{\sv_l - \mu z}\Big| = \Big|\frac{2l+1}{2l+1-2\mu}\Big| \leq 2\mu + 1.
\]
Then 
\[
    |f_{\sv_l}(\bq,z;d,\mu)| \leq g(\bq,z;d,\mu), \ \forall \ l\in\mathbb{Z}_{\geq 1},
\]
where 
\[
    g(\bq,z;d,\mu) := e^{|\log q_0|/z}\frac{q_1^dq_2^{d+\mu}}{d!(d+\mu)!z^{2d+\mu}}(2\mu + 1).
\]
The function $\sum_{d,\mu}g(\bq,z;d,\mu) < \infty$, so by Lebesgue's dominated convergence theorem, we get
\[
    I^2(\bq;\infty, z) = \sum_{d,\mu} f_\infty(\bq,z;d,\mu) = \sum_{d,\mu}\lim_{l\rightarrow \infty}f_{\sv_l}(\bq,z;d,\mu) =
    \lim_{l\rightarrow \infty} I^2(\bq;\sv_l,z).
\]
    
\

In step (c), we fix $\bq,z$, where $z\neq 0$.
Let 
\[
    a_{d,\mu}^k := \frac{(-q_2)^\mu\mu^kz^k}{(d+\mu)!z^\mu}.
\]
We first fix $d$ and $k$, and show $\sum_{\mu\geq 1}a_{d,\mu}^k$ is absolutely convergent. We have
\begin{equation*}
    \begin{aligned}
        \frac{|a_{d,\mu+1}^k|}{|a_{d,\mu}^k|} &= \frac{|q_2|^{\mu+1}(\mu+1)^k}{(d+\mu+1)!|z|^{\mu+1}}\cdot 
        \frac{(d+\mu)!|z|^\mu}{|q_2|^\mu\mu^k}
        \\
        &= \left|\frac{q_2}{z}\right|\frac{(1+1/\mu)^k}{d+\mu+1} \rightarrow 0 \text{ as } \mu\rightarrow 0.
    \end{aligned}
\end{equation*}
Therefore, there is a series of well-defined functions $\{A_d^k(q_2,z)\}_{d,k\geq 0}$ such that
\begin{equation*}
    \begin{aligned}
        &\sum_{\mu\geq 1} |a_{d,\mu}^k(q_2,z)| = A_d^k(q_2,z) < \infty,
        \\
        &|\varphi_k(\bq,z)| \leq e^{|(\log q_0)/z|}\sum_{d\geq 0} \Big|\frac{q_1^dq_2^d}{d!z^{2d}}\Big|A_d^k(q_2,z)
        \\
        & \quad \quad \quad \quad \leq e^{|(\log q_0)/z|}A_0^k(q_2,z)\sum_{d\geq 0} \Big|\frac{q_1^dq_2^d}{d!z^{2d}}\Big|.
    \end{aligned}
\end{equation*}
Let
\begin{equation*}
    \begin{aligned}b_d := \frac{q_1^dq_2^d}{d!z^{2d}},\quad
        \sqrt[d]{|b_d|} = \frac{1}{\sqrt[d]{d!}}\left|\frac{q_1q_2}{z^2}\right|\rightarrow 0 \text{\ as \ }d\rightarrow +\infty.
    \end{aligned}
\end{equation*}
Then we know $\varphi_k(\bq,z)$ is well-defined for all $\bq,z$. Furthermore, for fixed $\bq,z$ and for every $k$, we have
\[
    \varphi_{k+1}(\bq,z)\sv^{-k-1} = o(\varphi_k(\bq,z)\sv^{-k}) \ \text{ as  } \sv \rightarrow \infty.
\]
Hence, the series $\{\varphi_k(\bq,z)\sv^{-k}\}_{k=0}^\infty$ constitutes an asymptotic scale.

~

In step (d), assume $\bq,z>0$, we need to estimate the limit in Proposition \ref{prop:asym-limit}. Let
\begin{equation*}
    \begin{aligned}
            h_\sv(\bq,z;d,\mu) :=& \ \frac{q_1^dq_2^{d+\mu}(-1)^\mu}{d!(d+\mu)!z^{2d+\mu}}\sv^N\Big(\frac{\sv}{\sv-\mu z} - \sum_{k=0}^{N-1}\frac{\mu^kz^k}{\sv^k}\Big)
            \\
            =& \ \frac{q_1^dq_2^{d+\mu}(-1)^\mu}{d!(d+\mu)!z^{2d+\mu}}\frac{\sv}{\sv-\mu z}\mu^Nz^N,
            \\
            h_\infty(\bq,z;d,\mu) :=& \ \frac{q_1^dq_2^{d+\mu}(-1)^\mu}{d!(d+\mu)!z^{2d+\mu}}\mu^Nz^N.
    \end{aligned}
\end{equation*}    
Observe that $h_\sv(\bq,z;d,\mu)$ converges to $h_\infty(\bq,z;d,\mu)$  pointwisely, as $\sv$ tends to infinity.

Fix $z$ and let $\sv_l := (l+1/2)z$. We have
\begin{equation*}
    \begin{aligned}
        & \ \ \ \ \ |h_{\sv_l}(\bq,z;d,\mu)| = \frac{q_1^dq_2^{d+\mu}}{d!(d+\mu)!z^{2d+\mu}}\sv_l^N\Big|\frac{\sv_l}{\sv_l-\mu z} - \sum_{k=0}^{N-1}\frac{\mu^kz^k}{\sv_l^k}\Big|
        \\
        &= \frac{q_1^dq_2^{d+\mu}}{d!(d+\mu)!z^{2d+\mu}}\Big|\frac{\sv_l(\mu z)^N}{\sv_l - \mu z}\Big| \leq \frac{q_1^dq_2^{d+\mu}}{d!(d+\mu)!z^{2d+\mu}}(\mu z)^N(2\mu + 1).
    \end{aligned}
\end{equation*}

Notice that for every fixed $z$, the function 
\[
    e^{(\log q_0)/z} \sum_{d,\mu} \frac{q_1^dq_2^{d+\mu}}{d!(d+\mu)!z^{2d+\mu}}(\mu z)^N(2\mu + 1) < \infty.
\]
By Lebesgue's dominated convergence theorem, we have
\[
    \lim_{l\rightarrow \infty}  \sv_l^N\Big(I^2(\bq;\sv_l,z)-\sum_{k=0}^{N-1}\varphi_k(\bq,z)\sv_l^{-k}\Big) = e^{(\log q_0)/z}\sum_{d,\mu}h_\infty(\bq,z;d,\mu) =\varphi_N(\bq,z),
\]
i.e.
\[
    \lim_{l\rightarrow \infty} \frac{I^2(\bq;\sv_l,z)-\sum_{k=0}^{N-1}\varphi_k(\bq,z)\sv_l^{-k}}{\varphi_N(\bq,z)\sv_l^{-N}} = 1.
\]
Hence, $\{\varphi_k(\bq,z)\sv^{-k}\}_{k=0}^\infty$ is an asymptotic series of $I^2(\bq;\sv,z)$ for every fixed $\bq,z$ and well-chosen $\sv_l\rightarrow \infty$.

~

In step (e), we will show the $z^{-2}$-coefficient of the asymptotic series is well-defined.
In other words, we will show that $z^{-2}$-coefficient of $\varphi_k(\bq,z)$ is well-defined for all $k\in\bZ_{\geq 0}$.

We expand $\varphi_k(\bq,z)$ as formal series of $z$:
\begin{equation*}
    \begin{aligned}
        \varphi_k(\bq,z) &= \sum_{l\geq 0} \frac{(\log q_0)^l}{l!z^l}\sum_{d\geq 0, \mu >0}\frac{q_1^dq_2^{d+\mu}(-1)^\mu}{d!(d+\mu)!z^{2d+\mu}}\mu^kz^k,
        \\
        [z^{-m}]\varphi_k(\bq,z) &= \sum_{l+2d+\mu=k+m\atop {l,d\geq 0,\ \mu \geq 1}}  \frac{(\log q_0)^l}{l!}\frac{q_1^dq_2^{d+\mu}(-1)^\mu}{d!(d+\mu)!}\mu^k, \quad (m\in\bZ_{\geq 0}).
    \end{aligned}
\end{equation*}
Notice that $[z^{-m}]\varphi_k(\bq,z)$ is a finite sum, so it is well-defined. 

~

The same argument can be applied to $I^1(\bq;\sv,z)$.

\end{document}